\documentclass{amsart}
\usepackage[marginratio=1:1,height=584pt,width=390pt,tmargin=117pt]{geometry}
\usepackage[pdftex]{graphicx}
\usepackage{amssymb}
\usepackage{amsmath}
\usepackage{float}
\usepackage{placeins}
\usepackage[pdftex,dvipsnames]{xcolor} 
\usepackage{array}
\usepackage{calc}
\usepackage[ruled,noline]{algorithm2e}
\SetKwInOut{Input}{Input} 
\SetKwInOut{Output}{Output}
\usepackage[font=small,labelfont=bf]{caption}
\usepackage{subcaption}
\DeclareGraphicsExtensions{.eps}
\usepackage[outdir=./pictures/]{epstopdf}

\newcolumntype{?}{!{\vrule width 1.5pt}}

\def\BState{\State\hskip-\ALG@thistlm}

\usepackage[colorlinks=true,
           linkcolor=red,
            urlcolor=blue,
            citecolor=black]{hyperref}

\numberwithin{equation}{section}

\theoremstyle{plain}
\newtheorem{thm}{Theorem}[section]

\newtheorem{prop}[thm]{Proposition}

\theoremstyle{definition}

\theoremstyle{remark}
\newtheorem*{rem}{Remark}

\newcommand{\RR}{\mathbb{R}}

\newcommand*{\FF}[1]{\mathfrak{F}_{#1}} 
\newcommand*{\pp}[1]{\mathcal{P}_{\Omega_{#1}}} 

\newcommand*{\Rgf}[1]{\Gamma_R^{({#1})}} 

\newcommand*{\arr}{\longrightarrow}
\newcommand{\wh}[1]{\widehat{#1}}

\renewcommand{\Re}{\operatorname{Re}}             
\renewcommand{\Im}{\operatorname{Im}}             

\newcommand*{\ep}{\varepsilon}

\newcommand{\p}{\partial}
\newcommand{\ds}{\displaystyle}

\newcommand{\figref}[1]{\figurename~\ref{#1}}

\numberwithin{equation}{section}		
\numberwithin{figure}{section}			
\numberwithin{table}{section}				

\author{Andrea Scapin}
\address{Department of Mathematics, 
ETH Z\"urich, 
R\"amistrasse 101, CH-8092 Z\"urich, Switzerland.}
\email{andrea.scapin@sam.math.ethz.ch}
\title[Electrocommunication for weakly electric fish]{Electrocommunication for weakly electric fish}
\thanks{\footnotesize  
This work was supported by the SNF grant 200021-172483.}


\begin{document}

\begin{abstract} This paper addresses the problem of the electro-communication for weakly electric fish. In particular we aim at sheding light on
how the fish circumvent the jamming issue for both electro-communication and active electro-sensing. A real-time tracking algorithm is presented.
\end{abstract}

\maketitle

\section{Introduction}

In this paper we address the problem of studying the behaviour of two weakly electric fish when they populate the same environment.
Those kind of animals orient themselves at night in complete darkness by using their active electro-sensing system. They generate a stable, relatively high-frequency, weak electric field and transdermically perceive the corresponding feedback by means of many receptors on its skin. Since they have an electric sense that allows underwater navigation, target classification and intraspecific communication, they are privileged animals for bio-inspiring man-built autonomous systems \cite{caputi,p15,donati,helli,hoshi,p3,wwang,zheng}. 
For electro-communication purposes, in processing sensory information, this system has to separate feedback associated with their own signals from interfering sensations caused by signals from other animals. Wave and pulse species employ different mechanisms to minimize interference with EODs of conspecifics. It has been observed that certain wave species, having wave-type electric organ discharge (EOD) waveforms, such as Eigenmannia and Gymnarchus, reflexively shift their EOD frequency away from interfering frequencies of nearby conspecifics, in order to avoid ``jamming" each others electrical signals.
This phenomenon is known as \emph{jamming avoidance response} (JAR) \cite{bullock,helli2,helli3}. The electro-communication for the weakly electric fish has already been studied in the case of a simplified model consisting of a dipole-dipole interaction \cite{wwang}.

A lot of effort has also been devoted to the electro-sensing problem, that is, the capability of the animal to detect and recognize a dielectric target nearby \cite{p1,p4,p5,p6,p7,p8,p9,p10,p11,p12,p13,p14}. For the mathematical model of the electric fish described in \cite{Am2}, it has been shown that the fish is able to locate a small target by employing a MUSIC-type algorithm based on a multi-frequency approach. Its robustness with respect to
measurement noise and its sensitivity with respect to the number
of frequencies, the number of sensors, and the distance to the
target have also been illustrated.
The classification capabilities of the electric fish have also been investigated. In particular, invariant quantities under rotation, translation and scaling of the target, based on the generalized polarization tensors (GPTs), have been derived and used to identify a small homogeneous target among other shapes of a pre-fixed dictionary \cite{Am3,Am5,Am6,Am8,Am9}. The stability of the identifying procedure has been discussed. The recognition algorithm has been recently extended to sense small inhomogeneous target \cite{scapin}.
In \cite{maciver}, a capacitive sensing method has recently been implemented. It has been shown that the size of a capacitive sphere can be estimated from multi-frequency electrosensory data.
In \cite{faouzi}, uniqueness and
stability estimates to the considered electro-sensing inverse problem have been established.

In the present work we designed and implemented a real-time tracking algorithm for a fish to track another conspecific that is swimming nearby. In particular, we showed  that the following fish can sense the presence of the leading fish and can estimate its positions by using a MUSIC-type algorithm for searching its electric organ. We also showed that the fish can locate a small dielectric target which lies in its electro-sensing range even when another fish is swimming nearby, by filtering out its interfering signal and by applying the MUSIC-type algorithm developed in \cite{Am2}.



The paper is organized as follows. In Section 2, starting from Maxwell's equations in time domain we adapt the mathematical model of the electric fish proposed in \cite{Am2} in order to be able to consider many fish with EOD working at possibly different frequencies. We give a decomposition formula for the potential and, as a consequence, we decouple the dipolar signals of the two fish when they have different EOD fundamental frequencies. The amplitude of each signal can be retrieved from the measurements using Fourier analysis techniques.

In Section 3, we use the decomposition formula for the total signal to tackle the problem in the frequency domain. This allows us to employ a non-iterative MUSIC-type dipole search algorithm for a fish to track another fish of the same species nearby. 


In Section 4, we provide a method for a fish to electro-sense a small dielectric target in the presence of many conspecifics. The aim of this section is to locate the target which making use of the dipolar approximation of the transdermal potential modulations. We show that the multi-frequency MUSIC-type algorithm in \cite{Am6} is still applicable after decomposing the total signal.

In Section 5, many numerical simulations are driven. The performances of the real-time tracking algorithm are reported. We show that the algorithms work well even when the measurements are corrupted by noise.

\section{The two-fish model and the jamming avoidance response}

\noindent Let $\Omega$ be a simply-connected bounded domain. We assume $\Omega \in C^{2, \alpha}$ for some $0 < \alpha < 1$. Given an arbitrary function $w$ defined on $\RR^2 \setminus \p \Omega$ and $x \in \p \Omega$,  we define
\begin{align*} w(x)|_{\pm} & := \lim_{t \to 0} w(x \pm  t\nu(x)) ,\\
\frac{\p w}{\p \nu} (x) \biggr |_{\pm} & := \lim_{t \to 0} \nabla w(x \pm  t\nu(x)) \cdot \nu(x) ,
\end{align*} 
where $\nu$ is the outward normal to $\p \Omega$.

\noindent Let us denote by $\Gamma$ the fundamental solution of the Laplacian in $\RR^2$, that is,
\begin{equation*} \Gamma(x - y) = \frac{1}{2\pi} \log \| x - y\| , \quad x\ne y \in \RR^2.\end{equation*}
The single- and double-layer potentials on $\Omega$, $\mathcal{S}_{\Omega}$ and $\mathcal{D}_{\Omega}$, are the operators that respectively map any $\phi \in L^2(\p \Omega)$ to  
\begin{align*}
\mathcal{S}_{\Omega}[\phi] (x) &= \int_{\p \Omega} \Gamma (x,y) \phi (y) \mbox{ d} s_y ,\\
\mathcal{D}_{\Omega}[\phi] (x) &= \int_{\p \Omega} \frac{\p \Gamma}{\p \nu_y}(x, y) \phi (y) \mbox{ d} s_y  .
\end{align*}
Recall that for $\phi \in L^2(\p \Omega)$, the functions $\mathcal{S}_{\p \Omega}[\phi]$ and $\mathcal{D}_{\p \Omega}[\phi]$ are harmonic functions in $\RR^2 \setminus \p \Omega$.

\noindent The behaviour of these functions across the boundary $\p \Omega$ is described by the following  relations \cite{Am7}:
\begin{align*}
\mathcal{S}_{\Omega}[\phi] |_+ &= \mathcal{S}_{\Omega}[\phi] |_- ,\\
\frac{\partial \mathcal{S}_{\Omega}[\phi]}{\partial \nu} \biggr |_{\pm} &=  \left ( \pm \frac{1}{2} I + \mathcal{K}_{\Omega}^* \right ) [\phi] ,\\
\mathcal{D}_{\Omega}[\phi] |_{\pm} &=  \left ( \mp \frac{1}{2} I + \mathcal{K}_{\Omega} \right ) [\phi]  ,\\
\frac{\partial \mathcal{D}_{\Omega}[\phi]}{\partial \nu} \biggr |_{+}  & = \frac{\partial \mathcal{D}_{\Omega}[\phi]}{\partial \nu} \biggr |_{-} .
\end{align*}
The operator $\mathcal{K}_{\Omega}$ and its $L^2$-adjoint $\mathcal{K}_{\Omega}^*$ are given by the following formulas:
\begin{align}
\mathcal{K}_{\Omega}[\phi] (x) &:= \frac{1}{2 \pi}  \int_{\p \Omega} \frac{(y - x) \cdot \nu(y)}{|x - y|^2} \phi (y) \mbox{ d} s_y , \qquad x \in \p \Omega,\\
\mathcal{K}_{\Omega}^*[\phi] (x) &:= \frac{1}{2 \pi}  \int_{\p \Omega}\frac{(x - y) \cdot \nu(x)}{|x - y|^2} \phi (y) \mbox{ d} s_y , \qquad x \in \p \Omega .
\end{align}

\medskip

\noindent For the sake of simplicity, we consider the case of two weakly electric fish $\FF{1}$ and $\FF{2}$. The extension to the case of many fish is immediate.

\medskip

\noindent Starting from Maxwell's equations in time domain we derive
\[ \nabla \cdot ( \sigma + \ep \p_t ) E = - \nabla \cdot j_s \;\quad \mbox{in } \RR^2 ,\]
where $\sigma$ is the conductivity of the medium, $\ep$ is the electric permittivity, $E$ is the electric field, $j_s$ is a source of current.
Let $\omega_1$ and $\omega_2$ be the fundamental frequencies associated to the oscillations of the electric organ discharge (EOD) of the two fish $\FF{1}$ and $\FF{2}$, respectively. We consider a source term which is of the form
\[ - \nabla \cdot j_s = e^{i\omega_1 t} f_1(x) + e^{i\omega_2 t} f_2(x) ,\]
where $f_1 = \sum \alpha_j^{(1)} \delta_{x_j^{(1)}}$ and $f_2 = \sum \alpha_j^{(2)} \delta_{x_j^{(2)}}$ are the spatial dipoles located inside $\Omega_1$ and $\Omega_2$, respectively. Throughout this paper we assume that the dipoles $f_1$ and $f_2$ satisfy the local charge neutrality conditions:
\begin{equation*} \alpha_1^{(i)} + \alpha_2^{(i)} = 0 \quad \mbox{ for } i = 1\, ;\end{equation*}
see \cite{Am2}.
Considering the boundary conditions as in \cite{Am2}, we get the following system of equations:
\begin{equation} \label{eq:model_u_timedomain} \begin{cases} \Delta u(x,t) = f_1(x) h_1(t)  & \mbox{in } \Omega_1\times \RR^+ ,\\ \Delta u(x,t) = f_2(x) h_2(t)  & \mbox{in } \Omega_2 \times \RR^+ ,\\ \nabla  \cdot (\sigma(x) + \ep(x) \p_t  )  \nabla  u(x,t) = 0  &  \mbox{in } (\RR^2 \setminus \overline{\Omega_1 \cup \Omega_2} ) \times \RR^+ ,\\ u|_+ - u |_- = \xi_1 \dfrac{\partial u}{\partial \nu} \biggr |_+  &  \mbox{on } \partial \Omega_1 \times \RR^+, \\ u|_+ - u |_- = \xi_2 \dfrac{\partial u}{\partial \nu} \biggr |_+  &  \mbox{on }  \p \Omega_2 \times \RR^+ ,\\ \dfrac{\partial u}{\partial \nu} \biggr |_- = 0   & \mbox{on } \partial \Omega_1 \times \RR^+ ,  \p \Omega_2 \times \RR^+ , \\ |u(x,t)| = O(|x|^{-1})  & \mbox{as } |x| \to \infty, \; t \in \RR^+,\end{cases} \end{equation}
where $\sigma_0, \ep_0$ are the material parameters of the target $D$, and $\xi_1$ and $\xi_2$ are the effective skin thickness parameters of $\mathfrak{F}_1$ and $\mathfrak{F}_2$, respectively. Here, $h_1$ and $h_2$ encode the type of signal generated by the fish.

\subsection{Wave-type fish}

For the wave-type fish we have $h_1(t) =  e^{i \omega_1 t}$ and $h_2(t) = e^{i \omega_2 t}$.

\noindent When $\omega_1 \ne \omega_2$ the overall signal is the superposition of two periodic signals oscillating at different frequencies.

\begin{prop} \label{prop:submodalities} If $\omega_1 \ne \omega_2$ such that $\omega_1 , \omega_2 \ne 0$, then the solution $u$ to the equations \eqref{eq:model_u_timedomain} can be represented as
	\begin{equation} \label{eq:representation_u}  u(x,t) = u_1(x) e^{i \omega_1 t} + u_2(x) e^{i \omega_2 t} ,\end{equation}
	where $u_1 , u_2$ satisfy the following transmission problems:
	\begin{equation} \label{eq:model_u1} \begin{cases} \Delta u_1(x) = f_1(x)   & \mbox{in } \Omega_1 ,\\ \Delta u_1 (x) = 0 & \mbox{in } \Omega_2 ,\\ \nabla  \cdot (\sigma(x) + i \omega_1 \ep(x)  )  \nabla  u_1(x) = 0  &  \mbox{in } \RR^2 \setminus \overline{\Omega_1 \cup \Omega_2}   ,\\ u_1|_+ - u_1 |_- = \xi_1 \dfrac{\partial u_1}{\partial \nu} \biggr |_+  &  \mbox{on } \partial \Omega_1 ,\\ u_1|_+ - u_1 |_- = \xi_2 \dfrac{\partial u_1}{\partial \nu} \biggr |_+  &  \mbox{on }  \p \Omega_2 ,\\ \dfrac{\partial u_1}{\partial \nu} \biggr |_- = 0   & \mbox{on } \partial \Omega_1  ,  \p \Omega_2  , \\ |u_1(x)| = O(|x|^{-1})  & \mbox{as } |x| \to \infty,\end{cases} \end{equation}
	and
	\begin{equation} \label{eq:model_u2} \begin{cases} \Delta u_2(x) = 0  & \mbox{in } \Omega_1 ,\\ \Delta u_2 (x) = f_2(x) & \mbox{in } \Omega_2  ,\\ \nabla  \cdot (\sigma(x) + i \omega_2 \ep(x)  )  \nabla  u_2(x) = 0  &  \mbox{in } \RR^2 \setminus \overline{\Omega_1 \cup \Omega_2}  ,\\ u_2|_+ - u_2 |_- = \xi_1 \dfrac{\partial u_2}{\partial \nu} \biggr |_+  &  \mbox{on } \partial \Omega_1 ,\\ u_2|_+ - u_2 |_- = \xi_2 \dfrac{\partial u_2}{\partial \nu} \biggr |_+  &  \mbox{on } \p \Omega_2 ,\\ \dfrac{\partial u_2}{\partial \nu} \biggr |_- = 0   & \mbox{on } \partial \Omega_1  ,  \p \Omega_2  , \\ |u_2(x)| = O(|x|^{-1})  & \mbox{as } |x| \to \infty.\end{cases} \end{equation}
	
\end{prop}

\begin{proof} We substitute \eqref{eq:representation_u} into \eqref{eq:model_u_timedomain}. Considering the equation in $\Omega_1 \times \RR^+$ we get
	\[e^{i \omega_1 t} \Delta u_1  + e^{i \omega_2 t} \Delta u_2 = e^{i \omega_1 t} f_1  .\]
	Thus
	\[ ( \Delta u_1 - f_1)  + e^{i (\omega_2 - \omega_1) t} \Delta u_2 =  0 ,\]
	which yields $\Delta u_1 - f_1 = 0$  in $\Omega_1$ and $\Delta u_2 = 0$ in $\Omega_1$.
	
	\noindent In the same manner, we get the equations satisfied by $u_1$ and $u_2$ in $\Omega_2$.
	
	\noindent Outside the fish bodies, we have
	\[ \nabla  \cdot (\sigma +  \ep \p_t )  \nabla  e^{i \omega_1 t} u_1 + \nabla  \cdot (\sigma +\ep \p_t  )  \nabla  e^{i \omega_2 t} u_2= 0  ,  \]
	\[ e^{i \omega_1 t} \nabla  \cdot (\sigma +  i \omega_1 \ep  )  \nabla  u_1 +  e^{i \omega_2 t} \nabla  \cdot (\sigma + i \omega_2 \ep  )  \nabla   u_2 = 0  ,  \]
	that yields $\nabla  \cdot (\sigma +  i \omega_1 \ep  )  \nabla  u_1 = 0$ in $\RR^2 \setminus (\overline{\Omega_1 \cup \Omega_2})$ and $\nabla  \cdot (\sigma +  i \omega_2 \ep  )  \nabla  u_2 = 0$ in $\RR^2 \setminus (\overline{\Omega_1 \cup \Omega_2})$.
	
	\noindent Finally it is easy to check that the boundary conditions remain unchanged because the time dependency does not appear explicitly.
	
\end{proof}

\begin{rem} 
	The potentials $u_1$ and $u_2$, that respectively solve \eqref{eq:model_u1} and \eqref{eq:model_u2}, have a meaningful interpretation that is based on two different sub-modalities of the electroreception. As a matter of fact, $u_1$ can be viewed as the potential when the fish $\FF{1}$ is active and $\FF{2}$ is passive, whereas $u_2$ can be viewed as the potential when the fish $\FF{1}$ is passive and $\FF{2}$ is active. See \cite{caputi2}.
	
	
\noindent	Formula \eqref{eq:representation_u} tells us that it is possible to study the total field looking separately at these two different oscillating regimes.
\end{rem}

\noindent The idea is to separate the two signals from the measurements of their superposition. This can be done easily by using signal analysis techniques, see \cite{lund}.

\noindent \figref{fig:jamming} illustrates the potential before the jamming avoidance response, when the fish emit signals at a certain common frequency,
whereas  \figref{fig:submodalities}  depicts the two submodalities contained in the total signal $u(x,t)$ after they have switched their EOD frequencies.

\subsection{Pulse-type fish}

For the pulse-type fish we have that $h_1(t)$ and $h_2(t)$ are pulse wave. We can assume that they both can be obtained from a standard pulse shape $h(t)$ (see Figure \figref{fig:standard_pulse}) by means of translation and scaling, i.e.,
\begin{equation*} h_1(t) = h(\eta_1 t - T_1) ,\end{equation*}
\begin{equation*} h_2(t) = h(\eta_2 t - T_2) .\end{equation*}

\begin{figure}[h]
	\centering
	\includegraphics[scale=0.55]{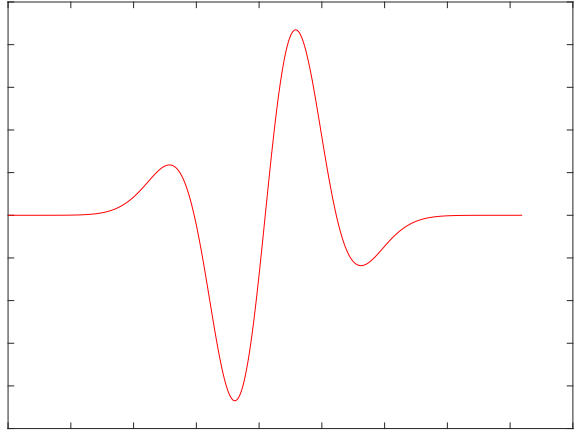}
	\caption{Standard shape of the pulse wave $h(t)$.}
	
	\label{fig:standard_pulse}
\end{figure}

For some pulse-type species, as Gymnotoid, the jamming avoidance response is obtained by shortening the duration of the emitted pulse, see \cite{helli3}. In this way, they minimize the chance of pulse coincidence by transient accelerations (decelerations) of their EOD rate. For $\eta_1 , \eta_2 > 0$ large enough such that $\mbox{supp}(h_1) \cap \mbox{supp}(h_2) = \emptyset$.

Thus, for $t_1, t_2 > 0$ such that $h_1(t_1) \ne 0$ and $h_2(t_2) \ne 0$ we can consider $u_1(x) := u(x,t_1)$ and $u_2(x) := u(x,t_2)$. These time-slices have the following property:

\begin{equation*}\begin{cases}\Delta u_1(x) = f_1(x) h_1(t_1), &x \in \Omega_1 \\ \Delta u_1(x) = 0  , &x \in \Omega_2,\end{cases} , \quad  \begin{cases}\Delta u_2(x) = 0, &x \in \Omega_1 \\ \Delta u_2(x) = f_2(x) h_2(t_2) , &x \in \Omega_2 \end{cases}. \end{equation*}
Hence we can achieve a separation of signals.

\begin{figure}[h]
	\begin{subfigure}[t]{0.5\textwidth}
		\centering
		\includegraphics[scale=0.40]{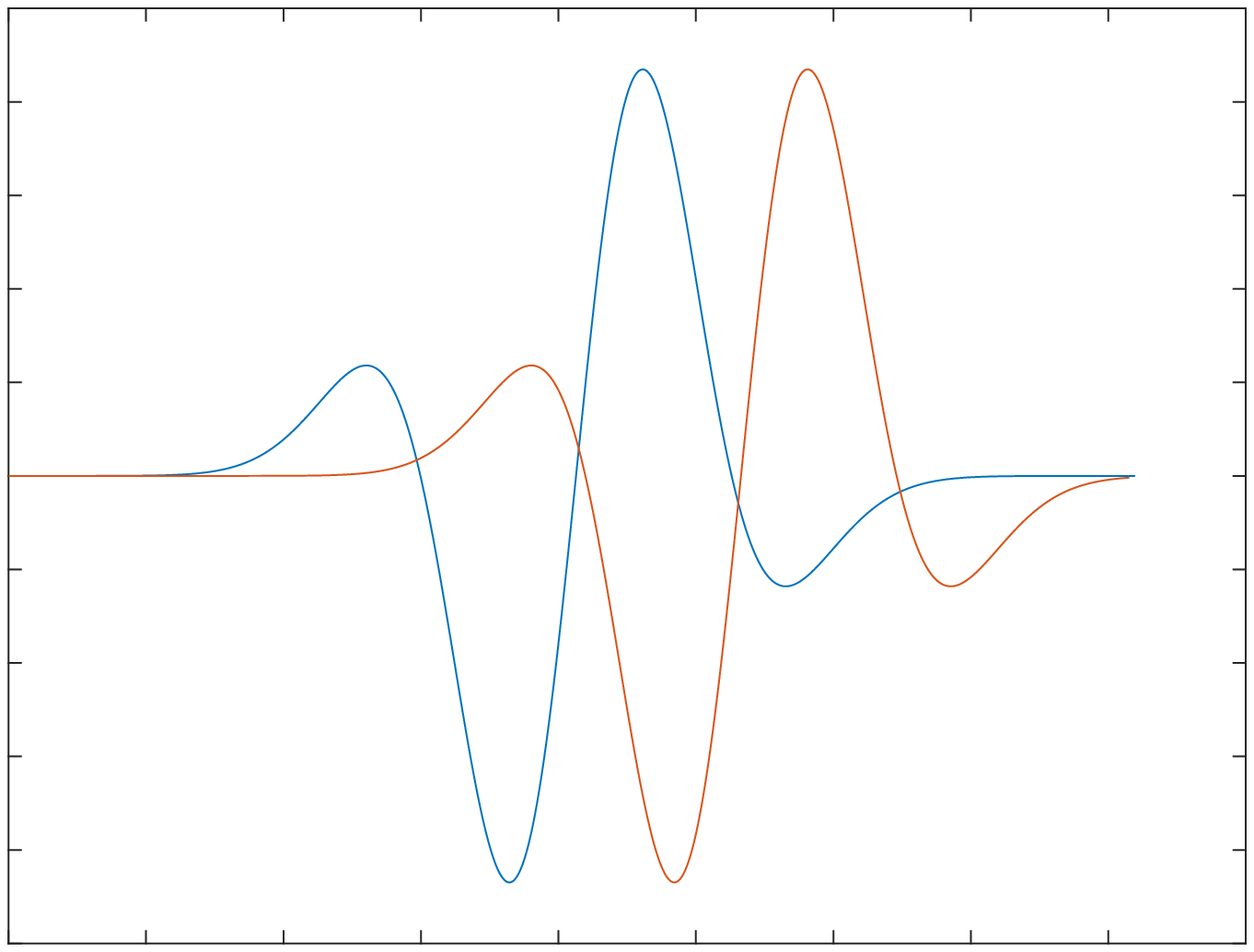}
		\caption{Before the JAR the two pulse signals may interfere with each other.}
	\end{subfigure}
	~ \hspace{0.03\textwidth}
	\begin{subfigure}[t]{0.5\textwidth}
		\centering
		\includegraphics[scale=0.40]{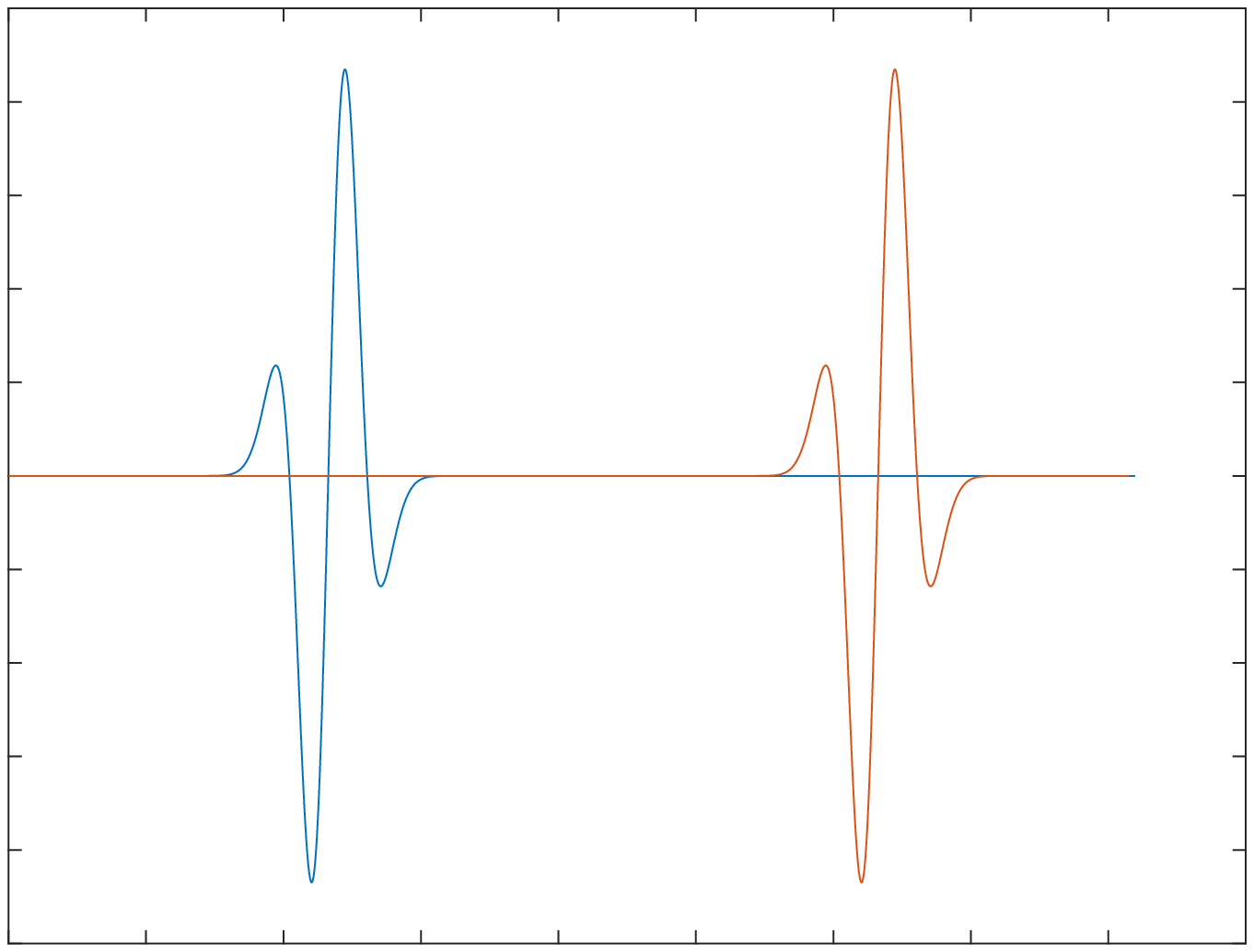}
		
		\caption{By shortening the duration of the pulse it is possible to identify two non-overlapping time-windows $I_1$ and $I_2$ corresponding to the signal emitted by the fish $\FF{1}$ and $\FF{2}$, respectively.}
	\end{subfigure}
	\caption{}
	\label{fig:submodalities}
\end{figure}

\noindent In the next sections, we will see an important consequence of Proposition \ref{prop:submodalities}. As a matter of fact $\mathfrak{F}_1$ can track $\mathfrak{F}_2$ by using the measurements of $u_2 |_{\p \Omega_1}$, solution to \eqref{eq:model_u2}, and can detect a small target $D$ by using the measurements of $u_1 |_{\p \Omega_1}$.

\begin{figure}[h]
        \centering
        \includegraphics[scale=0.55]{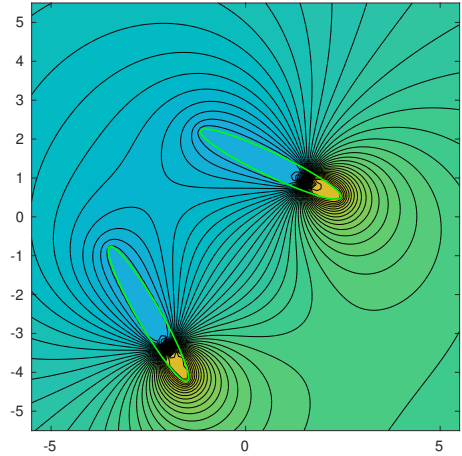}
        \caption{Before the JAR (the EOD frequencies of the two fish are the same). Plot of $u(x) = u_1(x) + u_2(x)$, where $u(x,t) = u(x) e^{i\omega_0 t}$.}

	\label{fig:jamming}
\end{figure}

\begin{figure}[h]
    \begin{subfigure}[t]{0.5\textwidth}
        \centering
        \includegraphics[scale=0.55]{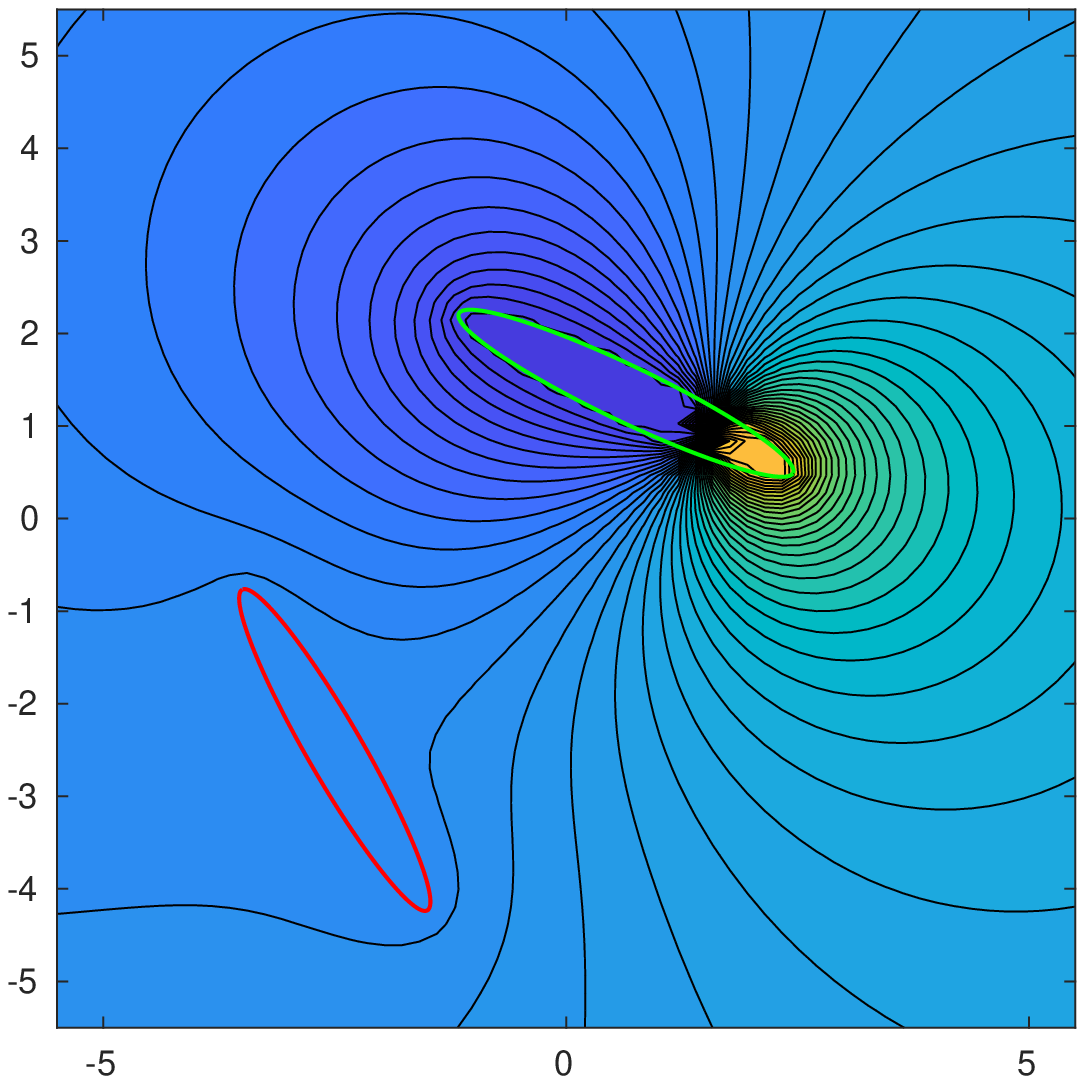}
        \caption{Plot of $u_2$. $\Omega_1$ (red) is passive and $\Omega_2$ (green) is active.}
    \end{subfigure}
    ~ \hspace{0.03\textwidth}
    \begin{subfigure}[t]{0.5\textwidth}
        \centering
        \includegraphics[scale=0.55]{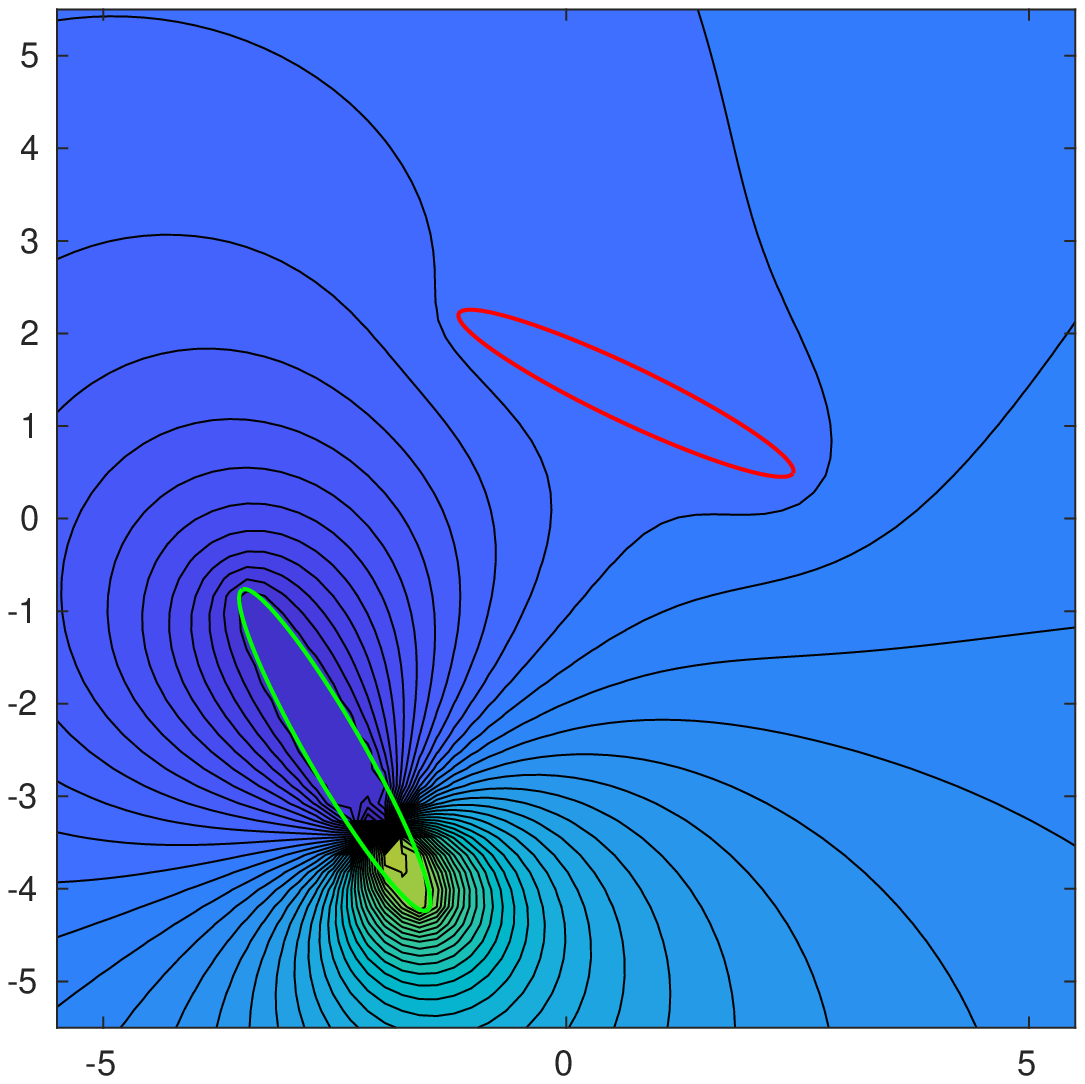}

        \caption{Plot of $u_1$. $\Omega_1$ (green) is active and $\Omega_2$ (red) is passive.}
        
        \label{fig:submodalitiesB}
    \end{subfigure}
      \caption{After the JAR (the EOD frequencies $\omega_1$ and $\omega_2$ of the two fish are apart from each other).}
	\label{fig:submodalities}
\end{figure}

\newpage
\section{Electro--communication}

The aim of this section is to give a mathematical procedure to model the communication abilities of the weakly electric fish, i.e., the capability of a fish to perceive a conspecific nearby. Assume, for instance, the point of view of the fish $\mathfrak{F}_1$. 
We want $\mathfrak{F}_1$ to estimate some basic features of $\mathfrak{F}_2$,
such as the position of its electric organ. More importantly, by using subsequent estimates, we want to design an algorithm for $\mathfrak{F}_1$ to track $\mathfrak{F}_2$.

\medskip

\noindent For the sake of clarity, we consider the case without the small dielectric target. It is worth emphasizing that the presence of the target is not troublesome since its effect on the tracking procedure is negligible even when the fish are swimming nearby.

\medskip

\noindent When $\mathfrak{F}_1$ gets close to $\mathfrak{F}_2$, both $\mathfrak{F}_1$ and $\mathfrak{F}_2$ experiment the so-called jamming avoidance response and thus their electric organ discharge (EOD) frequencies switch. When the EOD frequencies $\omega_1$ and $\omega_2$ are apart from each other, Proposition \ref{prop:submodalities} can be applied.

\noindent Let $u_2$ be the solution to the transmission problem \eqref{eq:model_u2}.
As previously mentioned, the function $u_2$ can be extracted from the total signal $u(x,t)$ using signal analysis techniques.

\medskip

\noindent We define
\begin{equation*} H^{u_2}(x) = \left ( \mathcal{S}_{\Omega_1} - \xi_1 \mathcal{D}_{\Omega_1} \right )\left [ \frac{\p u_{2}}{\p \nu} \biggr |_+ \right ] (x) .\end{equation*}
Let us recall the following boundary integral representation: for each  $x \in \RR^2 \setminus ( \overline{\Omega_1 \cup \Omega_2})$,
\begin{align*}
(u_2 - H^{u_2})(x) &=  \int_{\partial \Omega_2}  \left (   \frac{\partial u_2}{\partial \nu} \biggr |_{+}(y) \; \Gamma (x,y)- \frac{\partial \Gamma}{\partial \nu_y}(x,y) \; u_2 |_{+}(y) \, \right )  \text{d} s_y \,.
\end{align*}
Making use of the Robin boundary condition on $\p \Omega_2$ and integration by parts yields
\begin{align*}
(u_2 - H^{u_2})(x) 
&=  \int_{\partial \Omega_2}  \left (   \frac{\partial u_2}{\partial \nu} \biggr |_{+}(y) \; \Gamma (x,y)- \frac{\partial \Gamma}{\partial \nu_y}(x,y) \; u_2 |_{+}(y) \, \right )  \text{d} s_y
 \\
&=  \int_{\partial \Omega_2}  \left (   \frac{\partial u_2}{\partial \nu} \biggr |_{+}(y) \; \Gamma (x,y)- \frac{\partial \Gamma}{\partial \nu_y}(x,y) \; \left ( \xi_2 \frac{\partial u_2}{\partial \nu} \biggr |_{+}(y) +  u_2 |_{-}(y) \right ) \, \right )  \text{d} s_y  
\\ 
&=  \int_{\partial \Omega_2}     \frac{\partial u_2}{\partial \nu} \biggr |_{+}(y) \; \left ( \Gamma (x,y)- \xi_2 \frac{\partial \Gamma}{\partial \nu_y}(x,y) \right )  \; \text{d} s_y  -   \int_{\partial \Omega_2}  \frac{\partial \Gamma}{\partial \nu_y}(x,y) u_2 |_{-}(y) \, \text{d} s_y 
\\ 
&=  \int_{\partial \Omega_2}     \frac{\partial u_2}{\partial \nu} \biggr |_{+} \; \left ( \Gamma - \xi_2 \frac{\partial \Gamma}{\partial \nu} \right )  \; \text{d} s  \pm \alpha \left [ \Gamma (x-x_1^{(2)}) - \Gamma (x-x_2^{(2)}) \right ]  \, .
\end{align*}
Therefore, we obtain
\begin{equation} \label{eq:surf_and_dipole} 
(u_2 - H^{u_2})(x) =  \int_{\partial \Omega_2}     \frac{\partial u_2}{\partial \nu} \biggr |_{+} \; \left ( \Gamma - \xi_2 \frac{\partial \Gamma}{\partial \nu} \right )  \; \text{d} s  \pm \alpha \left [ \Gamma (x-x_1^{(2)}) - \Gamma (x-x_2^{(2)}) \right ]  \, .
\end{equation}
Observe that, for $x$ away from the $x_1^{(i)}$, we can approximate $\Gamma (x-x_2^{(2)}) - \Gamma (x-x_1^{(2)}) $ as follows:
\[\left [ \Gamma (x-x_2^{(2)}) - \Gamma (x-x_1^{(2)})  \right ] \approx \pm \nabla \Gamma (x - x_1^{(2)} ) \cdot (x_2^{(2)} - x_1^{(2)}) = \frac{(x - x_1^{(2)}) \cdot (x_2^{(2)} - x_1^{(2)})}{\|x - x_1^{(2)}\|^2} .\]

\noindent Consider an array of receptors $(x_l)_{l=1}^{M}$ on $\p \Omega_1$. We aim at solving the inverse source problem of determining the dipole, of $\mathfrak{F}_2$ from the knowledge of the measurements on the skin of $\mathfrak{F}_1$:
\begin{equation} \left \{ (u_2 - H^{u_2})(x_l) \;: \mbox{ for } l = 1, ... , N \right \}.\end{equation}
In order to estimate the dipole, we assume that the following single-dipole approximation holds:
\begin{equation} \label{eq:dipole_est}  (u_2 - H^{u_2})(x_l) \approx  \frac{(x_l - \wh{z}) \cdot \wh{\mathbf{p}}}{\|x_l - \wh{z} \|^2}   ,\end{equation}
where $\wh{\mathbf{p}}$ and $\widehat{z}$ denote respectively the moment and the center of the equivalent dipolar source.

\begin{rem}  The single-dipole approximation \eqref{eq:dipole_est} is  an equivalent representation of a spread source. However, in the presence of several well-separated sources, such approximation is not trustworthy anymore \cite{he}. In the case of $P \ge 3$ conspecifics we would extract $u_1 , ... , u_P$ components from the total signal, and the single-dipole approximation remains applicable to each one of the components  $u_2 , ... , u_P$.
\end{rem}





\section{Electro--sensing}

\noindent Now, suppose to have $\FF{1} , \FF{2}$ as before and a target close to $\FF{1}$.

\noindent Let $u_1$ be the solution to the transmission problem \eqref{eq:model_u1}, that is,
\begin{equation*} \label{eq:model_u1_test} \begin{cases} \Delta u_1(x) = f_1(x)   & \mbox{in } \Omega_1 ,\\ \Delta u_1 (x) = 0 & \mbox{in } \Omega_2 ,\\ \nabla  \cdot (\sigma(x) + i \omega_1 \ep(x)  )  \nabla  u_1(x) = 0  &  \mbox{in } \RR^2 \setminus \overline{\Omega_1 \cup \Omega_2}   ,\\ u_1|_+ - u_1 |_- = \xi_1 \dfrac{\partial u_1}{\partial \nu} \biggr |_+  &  \mbox{on } \partial \Omega_1 ,\\ u_1|_+ - u_1 |_- = \xi_2 \dfrac{\partial u_1}{\partial \nu} \biggr |_+  &  \mbox{on }  \p \Omega_2 ,\\ \dfrac{\partial u_1}{\partial \nu} \biggr |_- = 0   & \mbox{on } \partial \Omega_1  ,  \p \Omega_2  , \\ |u_1(x)| = O(|x|^{-1})  & \mbox{as } |x| \to \infty,\end{cases} \end{equation*}
and let $U_1$ be the background solution, that solves the problem
\begin{equation*} \label{eq:model_u1_bg_test} \begin{cases} \Delta U_1(x) = f_1(x)   & \mbox{in } \Omega_1 ,\\ \Delta U_1 (x) = 0 & \mbox{in } \Omega_2 ,\\ \Delta  U_1(x) = 0  &  \mbox{in } \RR^2 \setminus \overline{\Omega_1 \cup \Omega_2}   ,\\ U_1|_+ - U_1 |_- = \xi_1 \dfrac{\partial U_1}{\partial \nu} \biggr |_+  &  \mbox{on } \partial \Omega_1 ,\\ U_1|_+ - U_1 |_- = \xi_2 \dfrac{\partial U_1}{\partial \nu} \biggr |_+  &  \mbox{on }  \p \Omega_2 ,\\ \dfrac{\partial U_1}{\partial \nu} \biggr |_- = 0   & \mbox{on } \partial \Omega_1  ,  \p \Omega_2  , \\ |U_1(x)| = O(|x|^{-1})  & \mbox{as } |x| \to \infty.\end{cases} \end{equation*}
Consider $\Rgf{1,2}$ the Green's function associated with Robin boundary conditions, that is defined for $x \in \RR^2 \setminus ( \overline{\Omega_1 \cup \Omega_2})$ by
\begin{equation} \label{eq:model_GR} \begin{cases} - \Delta_y \Rgf{1,2}(x,y) = \delta_x(y) , & y \in \RR^2 \setminus \overline{\Omega_1 \cup \Omega_2}, \\  \Rgf{1,2}(x,y) |_+ -  \xi_1 \dfrac{\partial  \Rgf{1,2}}{\partial \nu_x} (x,y) \biggr |_+ = 0  , & y \in \partial \Omega_1 , \\ \Rgf{1,2}(x,y) |_+ -  \xi_2 \dfrac{\partial  \Rgf{1,2}}{\partial \nu_x} (x,y) \biggr |_+ = 0  , & y \in \partial \Omega_2 , \\  \left | \Rgf{1,2}(x,y) + \frac{1}{2 \pi} \log |y| \right | = O(|y|^{-1}) , & |y| \to \infty.\end{cases} \end{equation}
Recall the following boundary integral equation: for each  $x \in \RR^2 \setminus ( \overline{\Omega_1 \cup \Omega_2 \cup D})$,
\begin{equation*}
(u_1 - U_1)(x) =  \int_{\partial D}  \left (   \frac{\partial u}{\partial \nu} \biggr |_{+}(y) \; \Rgf{1,2} (x,y)- \frac{\partial \Rgf{1,2}}{\partial \nu_y}(x,y) \; u |_{+}(y) \, \right )  \text{d} s_y ,
\end{equation*}
where $U_1$ is the background solution, i.e., the solution without the inhomogeneity $D$, when the only dipolar source lies inside the body of $\mathfrak{F}_1$, see \figref{fig:submodalitiesB}.
\begin{equation*}
(u_1 - U_1)(x) =  \frac{(k-1)}{k} \int_{\partial D}  \left (   \frac{\partial u}{\partial \nu} \biggr |_{+}(y) \; \Rgf{1,2} (x,y) \, \right )  \text{d} s_y .
\end{equation*}

Let $B$ be a bounded open set with characteristic size $1$. Assume that $D = z + \delta B$, i.e., $D$ is a target located at $z$ which has characteristic size $\delta$. With the same argument as in \cite{Am2}, we obtain the following small volume approximation.
\begin{thm}[Dipolar approximation] Suppose $\text{dist}(\partial \Omega_1 , z) \gg 1$ and $\delta \ll 1$. Then for any $x \in \partial \Omega_1$, 
\begin{equation}\label{thm:dip_exp} \left (  \frac{\partial u_1}{\partial \nu} - \frac{\partial U_1}{\partial \nu} \right ) (x) = - \delta^2 \nabla U_1(z)^T M (\lambda, B) \nabla_y \left ( \frac{\partial \Gamma_R^{(1,2)}}{\partial \nu}  \biggr |_+ \right ) (x, z)  + O(\delta^3), \end{equation}
where $T$ denotes the transpose, $M(\lambda , B) = (m_{ij})_{i,j \in\{1,2\}}$ is the first-order polarization tensor associated with $B$ and contrast  parameter $\lambda$, given by
\begin{equation}\label{eq:PT-fo} m_{ij}  = \int_{\partial B}  y_{i}   \left ( \lambda I - \mathcal{K}_B^* \right )^{-1} \left ( \frac{\partial x_j}{\partial \nu} \biggr |_{\partial B} \right ) (y) \text{ d} s_y . \end{equation}
\end{thm}
Note that, since the background potential is real, for $x \in \partial \Omega_1$ we have
\begin{equation}\label{eq:dipole_approx_Im} 
\Im \left ( \frac{\partial u_1}{\partial \nu}\right ) (x) \approx - \delta^2 \nabla U_1(z)^T \Im M (\lambda, B) \nabla_y \left ( \frac{\partial \Gamma_R^{(1,2)}}{\partial \nu}  \biggr |_+ \right ) (x, z) .
\end{equation}

This last step is crucial to locate the target because $U_1$ is only approximately known from the measurements and even a very small displacement in the location of $\FF{2}$ can cause an error on the background potential $U_1$, which is of the same order as the contribution of the target.

\noindent On the other hand, when $z$ is not too close to $\p \Omega_2$, the contribution of $\FF{2}$ contained into $\nabla U_1(z)$ is negligible.
 Therefore, we approximate $\nabla U_1(z) \approx \nabla \wh{U}_1(z)$, where $\wh{U}_1$ is solution to the problem

\begin{equation} \label{eq:model_hat_u1_bg_test}
 \begin{cases} \Delta \wh{U}_1(x) = f_1(x)   & \mbox{in } \Omega_1 ,\\ \Delta  \wh{U}_1(x) = 0  &  \mbox{in } \RR^2 \setminus \overline{\Omega_1}  ,\\ \wh{U}_1|_+ - \wh{U}_1 |_- = \xi_1 \dfrac{\partial \wh{U}_1}{\partial \nu} \biggr |_+  &  \mbox{on } \partial \Omega_1 ,\\ \dfrac{\partial \wh{U}_1}{\partial \nu} \biggr |_- = 0   & \mbox{on } \partial \Omega_1  , \\ |\wh{U}_1(x)| = O(|x|^{-1})  & \mbox{as } |x| \to \infty.\end{cases} \end{equation}

After post-processing \eqref{eq:dipole_approx_Im} using the following operator
\begin{equation*}
{\pp{1}}  =  \frac{1}{2} I - \mathcal{K}_{\Omega_1}^* - \xi \frac{\p \mathcal{D}_{\Omega_1}}{\p \nu} ,
\end{equation*}
see \cite{Am2}, we get
\begin{equation}\label{eq:dipole_approx_Im_pp} 
\pp{1} \left [ \Im \left ( \frac{\partial u_1}{\partial \nu}\right )  \right ] (x) \approx  \delta^2 \nabla \wh{U}_1(z)^T \Im M (\lambda, B) \nabla_y \left ( \frac{\partial \Gamma}{\partial \nu_x}  \biggr |_+ \right ) (x, z) , \quad x \in \p \Omega_1 .
\end{equation}
Therefore, as long as $\text{dist}(z, \p \Omega_2) \gg 0$, the leading order term of the post-processed measured data is not significantly affected by the presence of $\FF{2}$.

A MUSIC-type algorithm for searching the position $z$ and a least square method for recovering the imaginary part of the polarization tensor $M(\lambda, B)$ can be applied, see \cite{Am5}. 


%
%
%
%
%

\section{Numerical experiments}

With applications in robotics in mind, and for the sake of simplicity, we can assume that the two fish populating our testing environment share the same effective thickness $\xi$ and the same shape, which is an ellipse with semiaxes $a = 2$ and $b = 0.3$. Therefore no tail-bending has been taken into account.

\noindent For the numerical computations of the direct solutions to the transmission problems involved in the following simulations, we solved the boundary integral system of equations by relying on boundary element techniques. We adapted the codes in \cite{code} to our  framework, with many fish populating the same environment.

\subsection{Electro-communication}

We perform numerical simulations to show how $\FF{1}$ can locate the position and the orientation of $\FF{2}$ by using a modified version of MUSIC-type algorithm for searching the dipolar source. Firstly, let us observe that accuracy is not improved by using a multi-frequency approach when noisy measurements are considered. Instead, $\FF{1}$ can use a MUSIC-type algorithm based on movement in order to improve the accuracy in the detection algorithm. 
We use the approximation \eqref{eq:dipole_est}.

\noindent We consider $N_s$ positions. For each $s \in \{1, ... , N_s\}$ let us denote by $\mathfrak{F}_2^s$ the fish at the position $s$. On its skin there are $N_r$ receptors $\{x_n^s\}_{n = 1}^{N_r}$. 
For each $s = 1, ... , N_s$, we define the $M \times 2$ lead field matrix $A_s$ as
\begin{equation} \label{eq:transfer_mat}\mathbf{A}_s(z) =  \begin{bmatrix}\ds  \frac{x_{1,1}^s - z_1}{\|x_1^s - z\|^2} &\ds \frac{x_{1,2}^s - z_1}{\|x_1^s - z\|^2} \\ \vdots & \vdots \\\ds \frac{x_{M,1}^s - z_1}{\|x_M^s - z \|^2} &\ds \frac{x_{M,2}^s - z_2}{\|x_M^s - z\|^2} \end{bmatrix} .\end{equation}
Let $\mathbf{F}$ be the Multi-Static Response (MSR) matrix defined as follows
\begin{equation*}
\mathbf{F} = \begin{bmatrix}
(u_2 - H^{u_2})(x_1^1) & \dots & 
(u_2 - H^{u_2})(x_1^{N_s}) \\

\vdots & \ddots & \vdots \\

(u_2 - H^{u_2})(x_{N_r}^1) & \dots & 
(u_2 - H^{u_2})(x_{N_r}^{N_s}) \\
\end{bmatrix} .
\end{equation*}
Moreover, we assume that the acquired measurements are corrupted by noise, i.e.,
\begin{equation*}
\mathbf{F}_{noisy} = \mathbf{F} + \mathbf{X} ,
\end{equation*}
where $\mathbf{X} \sim \mathcal{N}(0, \sigma_{\text{noise}}^2)$ is a Gaussian random variable with mean $0$ and variance $\sigma_{\text{noise}}^2$. In our simulations we set the variance to:
\[ \sigma_{\text{noise}} = (\mathbf{F}_{max} - \mathbf{F}_{min}) \sigma_0 ,  \]
where $\sigma_0$ is a positive constant called noise level, and $\mathbf{F}_{max}$ and $\mathbf{F}_{min}$ are the maximal and the minimal coefficient in the MSR matrix $\mathbf{F}$. 

\medskip

\noindent Let $\mathbf{F}^{\mathfrak{R}}_{noisy}$ be the real part of $\mathbf{F}_{noisy}$. Let $\lambda_1 \ge \lambda_2 \ge ... \ge \lambda_{N_r}$ be the eigenvalues of $\mathbf{F}^{\mathfrak{R}}_{noisy}  \cdot ( \mathbf{F}^{\mathfrak{R}}_{noisy} )^T$ and let $\Phi_1 , ... , \Phi_{N_r}$ be the correspondent eigenvectors.  The first eigenvalue is the one associated to the signal source and the span of the eigenvector $\Phi_1$ is called the signal subspace. The other eigenvectors span the noise subspace.

\noindent As it is well known, the MUSIC algorithm estimates the location of the dipole by checking the orthogonality between  $\mathbf{A}_s(z)$ \eqref{eq:transfer_mat} and the noise subspace projector $\mathcal{P}_N$ \cite{mosher}. This can be done for each position $s$.

\noindent For this purpose, we shall use a modified version of the MUSIC localizer in \cite{sekihara}, by simply taking the maximum over the positions:
\begin{equation} \label{eq:music_localizer2}
\mathcal{I}_2(z) = \max_{s = 1, ... , N_s} \left (  \frac{1}{\lambda_{\min} (\mathbf{A}_s(z)^T \mathcal{P}_N \mathbf{A}_s(z), \mathbf{A}_s(z) \mathbf{A}_s(z)^T)} \right ) ,
\end{equation}
where $\lambda_{\min}( \cdot , \cdot )$ indicates the generalized minimum eigenvalue of a matrix pair.

\noindent We expect that the MUSIC localizer has a large peak at the location of the equivalent dipole we are searching for. 
Once an estimate $\wh{z}$ of the true location has been obtained, the dipole moment can be estimated by means of the following formula:
\begin{equation} \wh{\mathbf{p}}_{est} = ( \mathbf{A}(\wh{z})^T \mathbf{A}(\wh{z}) )^{-1} \mathbf{A}(\wh{z})^T \Phi_1 .\end{equation} 
i.e., the least-square solution to the linear system
\begin{equation}\label{eq:linear_sys} \Phi_1 =\mathbf{A}(\wh{z}) \wh{\mathbf{p}}.\end{equation} 
\begin{figure}[h]
	\centering
	\includegraphics[scale=0.55]{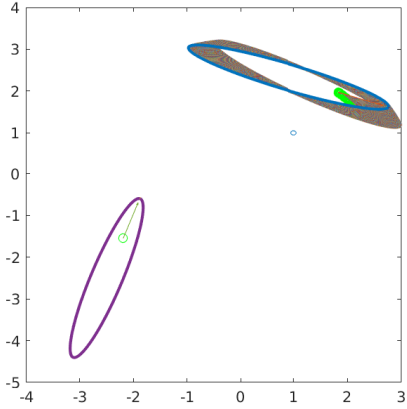}
	\caption{The setting. $\FF{1}$ is acquiring measurements at $N_s = 150$ different closely spaced positions.}
	
	\label{fig:movement}
\end{figure}
\begin{figure}[h]
	\centering
	\includegraphics[scale=0.55]{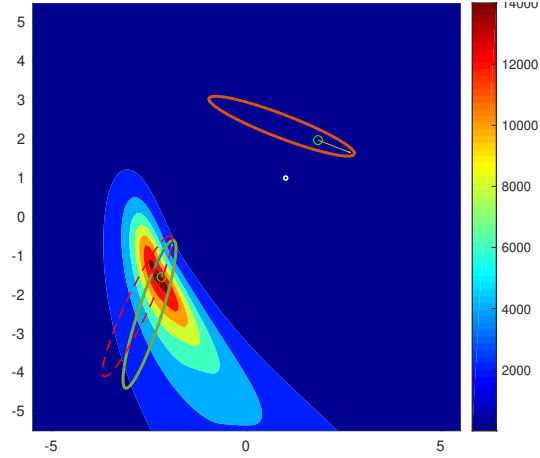}
	\caption{Estimate of the position and the orientation of $\FF{2}$, with noise level $\sigma_0 = 0.1$. The dashed red curve represents the estimated body of $\FF{2}$, whereas the green one represents the true body of $\FF{2}$. The white circle represents a small dielectric object placed between $\FF{1}$ and $\FF{2}$.}
	
	\label{fig:locate_fish}
\end{figure}

\begin{algorithm}[H]
	\Input{The feedback, that is the total electric potential signal $u(x,t)$ recorded by the receptors on $\p \Omega_1$.}
	
	\medskip
	
	\nl Decompose the feedback $u$ into $u_1$ and $u_2$ using signal separation techniques\;
	\nl $\texttt{MUSIC\_dipoleSearch}(u_2 |_{\p \Omega_1}, \Omega_1)$ :\\ \Indp
	\nl Build the (real part of the) MSR matrix $\mathbf{F}_{noisy} \in M(N_r \times N_s, \RR)$ from measurements collected during a short period\;
	\nl  Compute the eigen-decomposition of $\mathbf{F}_{noisy} \mathbf{F}_{noisy}^T = \Phi \Lambda \Phi^T$  and the noise subspace projector $\mathcal{P}_N$\;
	
	\nl Evaluate the MUSIC localizer $\mathcal{I}_2$ on the nodes of a fine uniform grid $\mathcal{G}$ in the vicinity of $\Omega_1$ \;
	
	\nl $\wh{z}  \leftarrow \; \arg \max_{\mathcal{G}} \; \mathcal{I}_2(z)$ \;

	\nl Determine $\wh{\mathbf{p}}_{est}$ as the least-square solution to the linear system \eqref{eq:linear_sys}\;
	\medskip
	\medskip
	\Output{an approximated position of the position of the electric organ of the  conspecific $\FF{2}$.}
	\caption{MS MUSIC: Detect the presence of a conspecific from skin measurements}
	\label{alg:detection_algorithm1}
\end{algorithm}

\subsection{Tracking}

Now we want to show that the dipole approximation that we assumed in the previous subsection is good enough to be used successfully for tracking purposes.

\noindent We assume the following setting for the numerical simulations. The fish $\FF{1}$ is swimming along a fixed trajectory. Let us assume that the motion of its electric organ is described by a continuous path $F : [t_1,t_N] \arr \RR^2$. Let $t_1 < t_2 < ... < t_N$ be a temporal grid on $[t_1,t_N]$ and let $t_j = s_1^j < ... < s_{M}^j = t_{j+1}$ be a grid on $[t_j,t_{j+1}]$ for $j=1 , ... , N-1$. 

\noindent At the beginning, when $t = t_1$, $\FF{2}$ starts following $\FF{1}$. The tracking is performed by estimating the positions of $\FF{1}$ at the nodes of the grid $t_1, ... , t_N$. Let us denote  by $X_n , Y_n$ and $p_n , q_n$ the positions and the orientations of $\FF{2}$ and $\FF{1}$ at $t = t_n$, respectively. In order to obtain an estimate $\wh{Y}_n$ of the position $Y_n$ we can apply Algorithm \ref{alg:detection_algorithm1}, that employs measurements at $s_1^{n-1} , ... , s_M^{n-1}$ to reduce the effect of the noise. More precisely, the discrete dynamic system that describes the evolution of the positions and orientations of the two fish is as follows:
\begin{equation}\begin{cases}
X_n = X_{n-1} + h_n p_{n-1}, \\
p_n = \mathbf{R}(\theta_n) p_{n-1},\\
Y_n = F(t_n) ,\\
q_n = F'(t_n) \approx \frac{F(t_n) - F(t_{n-1})}{h},
\end{cases}
\end{equation}
where $X_0$ and $p_0$ are the initial data. Let us define $T_{n-1} := \wh{Y}_{n-1} - X_{n-1}$, the pointing direction. The update of the orientation of $\FF{2}$ is given by an orthogonal matrix associated with a rotation by an angle $\theta_n$, $\mathbf{R}(\theta_n) \in O(2,\RR)$, and the turning angle is defined as
\begin{equation}
\theta_n :=  \theta_{n-1}  \pm \min \left ( \theta_{\max} , \wh{T_{n-1} p_{n-1}} \right ) .
\end{equation}
The numbers $h_1 , ... , h_M$ incorporate the velocity of the tracking fish and should be chosen adaptively, in order to allow a variety of maneuvering capabilities such as acceleration and deceleration, as well as swimming backwards when $h_n < 0$. In order to prevent both collision and separation, we shall assume the velocity to be a function of the distance between $X_n$ and $\wh{Y}_n$. 
$\theta_{\max}$ is the maximum turning angle. It is worth mentioning that the choice of $\theta_{\max}$  has a strong impact on the efficiency of the tracking procedure.

\begin{algorithm}[H]
	\Input{Temporal grid over $[t_1,t_N]$. The maximum turning angle $\theta_{\max}$.}
	
	\medskip
	
	\nl $\texttt{RT\_Tracking}(\theta_{\max},[t_1,t_N],N,M)$ :\\ \Indp
	\nl \For{$n\gets 1 , \dots , N$}{
		
		\nl	$X_n \gets X_{n-1} + h_n p_{n-1}$\;
		
		\nl $p_n \gets \mathbf{R}(\theta_n) p_{n-1}$\;
		
		\nl $\wh{Y}_n \gets \texttt{MUSIC\_dipoleSearch}(n)$  \\} \Indp
	\medskip
	\medskip
	\Output{Trajectory of the following fish.}
	\caption{Real-Time Tracking: Fish-follows-Fish algorithm.}
	\label{alg:detection_algorithm1}
\end{algorithm}

\begin{figure}[h]
	\centering
	
	\begin{subfigure}[t]{0.5\textwidth}
		\centering			
		\includegraphics[scale=0.5]{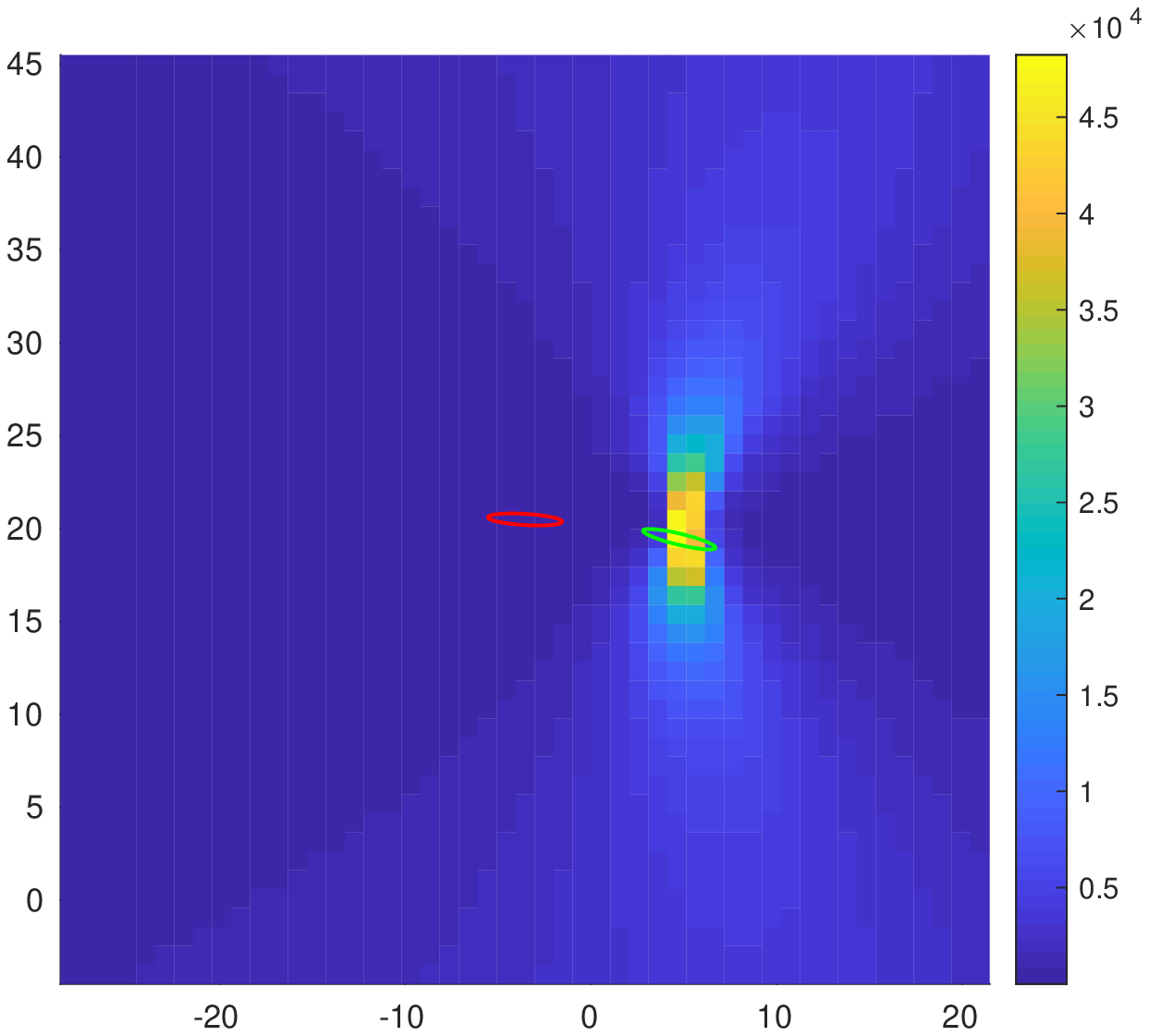}
		\caption{$\sigma_0 = 0.01$}
	\end{subfigure}
	~ \hspace{0.03\textwidth}
	\begin{subfigure}[t]{0.5\textwidth}
		\centering
		\includegraphics[scale=0.5]{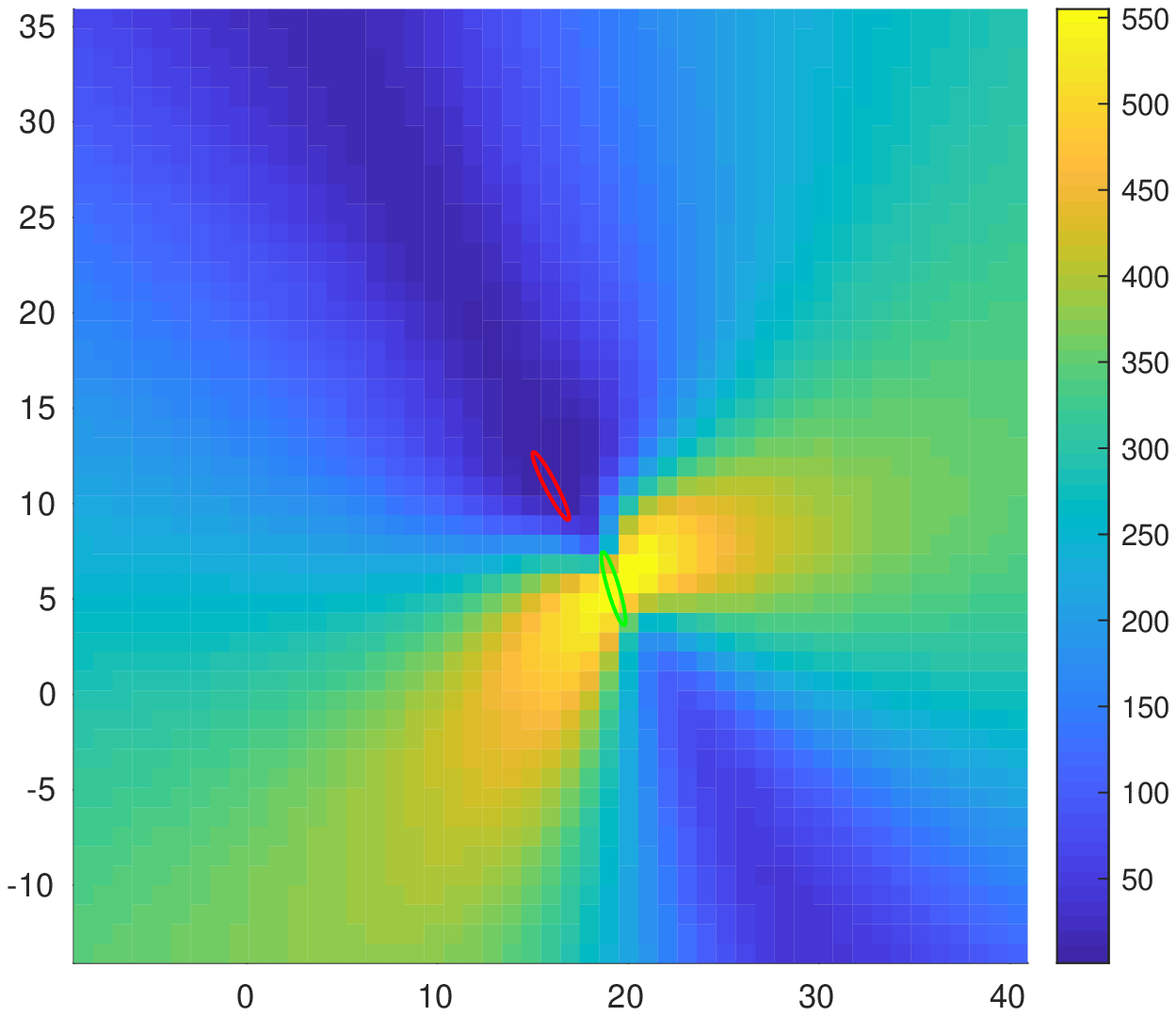}
		\caption{$\sigma_0 = 0.1$}
	\end{subfigure}
	~ \hspace{0.03\textwidth}\\
	
	%
	\caption{Plot of the imaging functional $\mathcal{I}_2$ that the fish $\FF{1}$ uses to track $\FF{2}$.}
	\label{fig:imagingfunc}
\end{figure}

\begin{figure}[h]
	\begin{subfigure}[t]{1\textwidth}
		\centering
		\includegraphics[scale=0.8]{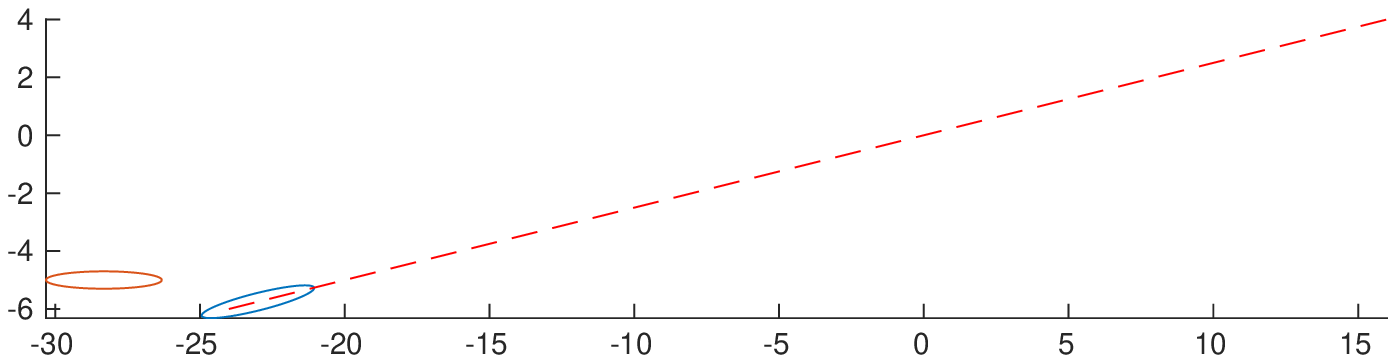}
		\caption{Starting position.}
	\end{subfigure}
	~ \hspace{0.03\textwidth}\\
	\begin{subfigure}[t]{0.8\textwidth}
		\centering
		\includegraphics[scale=0.75]{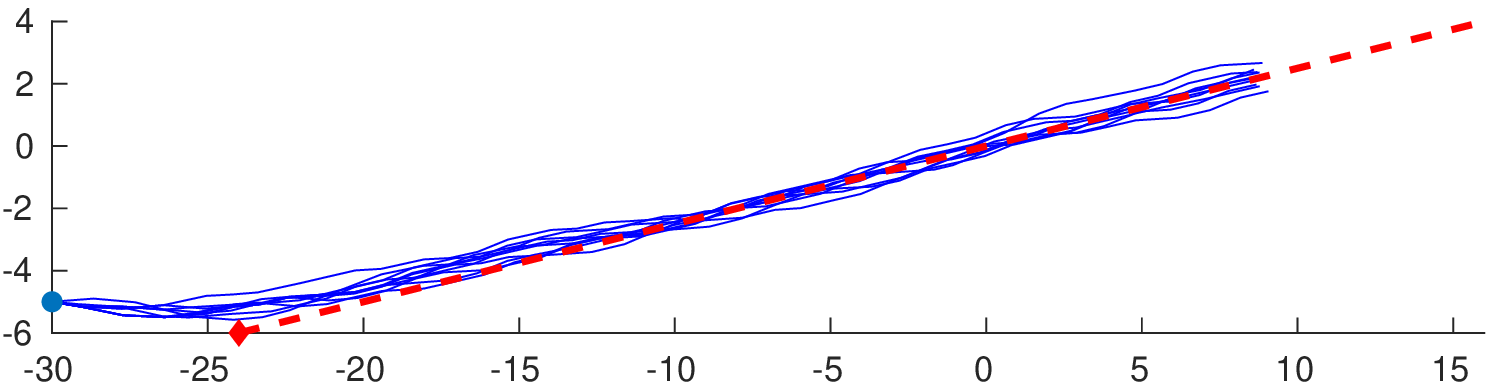}
		\caption{$\sigma_0 = 0.01$}
	\end{subfigure}
	~ \hspace{0.03\textwidth}\\
	\begin{subfigure}[t]{0.8\textwidth}
		\centering
		\includegraphics[scale=0.75]{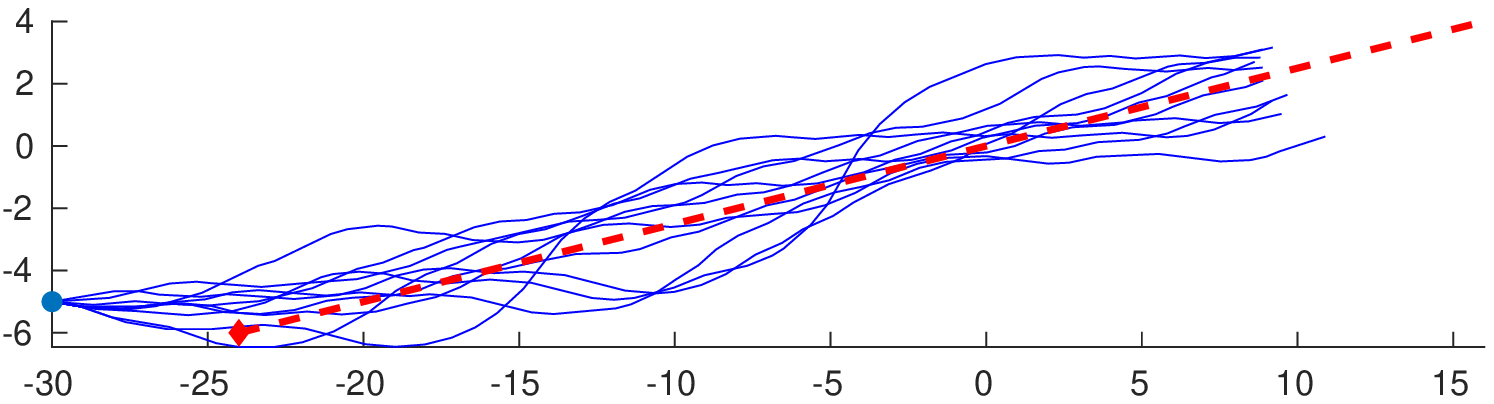}
		
		\caption{$\sigma_0 = 0.05$}
	\end{subfigure}
	~ \hspace{0.03\textwidth}\\
%
\begin{subfigure}[t]{0.8\textwidth}
	\centering
	\includegraphics[scale=0.75]{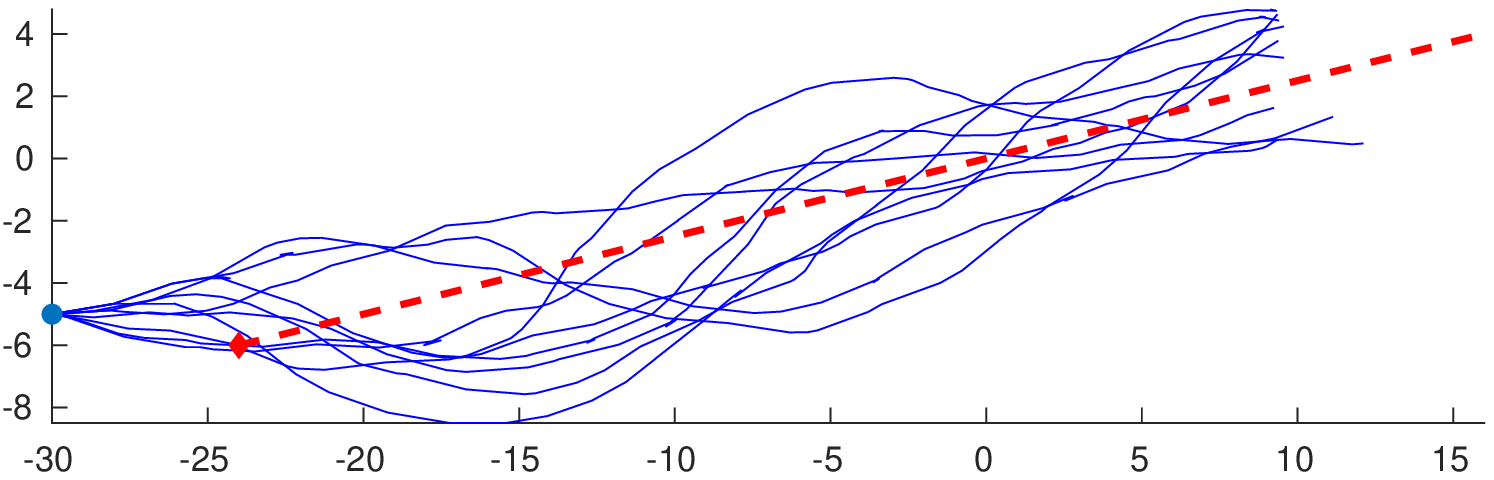}
	
	\caption{$\sigma_0 = 0.2$}
\end{subfigure}
%
	\caption{Plot of the linear trajectory tracking. $N_{\text{exp}} = 10$ trials have been considered.}
	\label{fig:traj_track_linear}
\end{figure}

\newpage

\begin{figure}[h]
	\begin{subfigure}[t]{0.5\textwidth}
		\centering
		\includegraphics[scale=0.6]{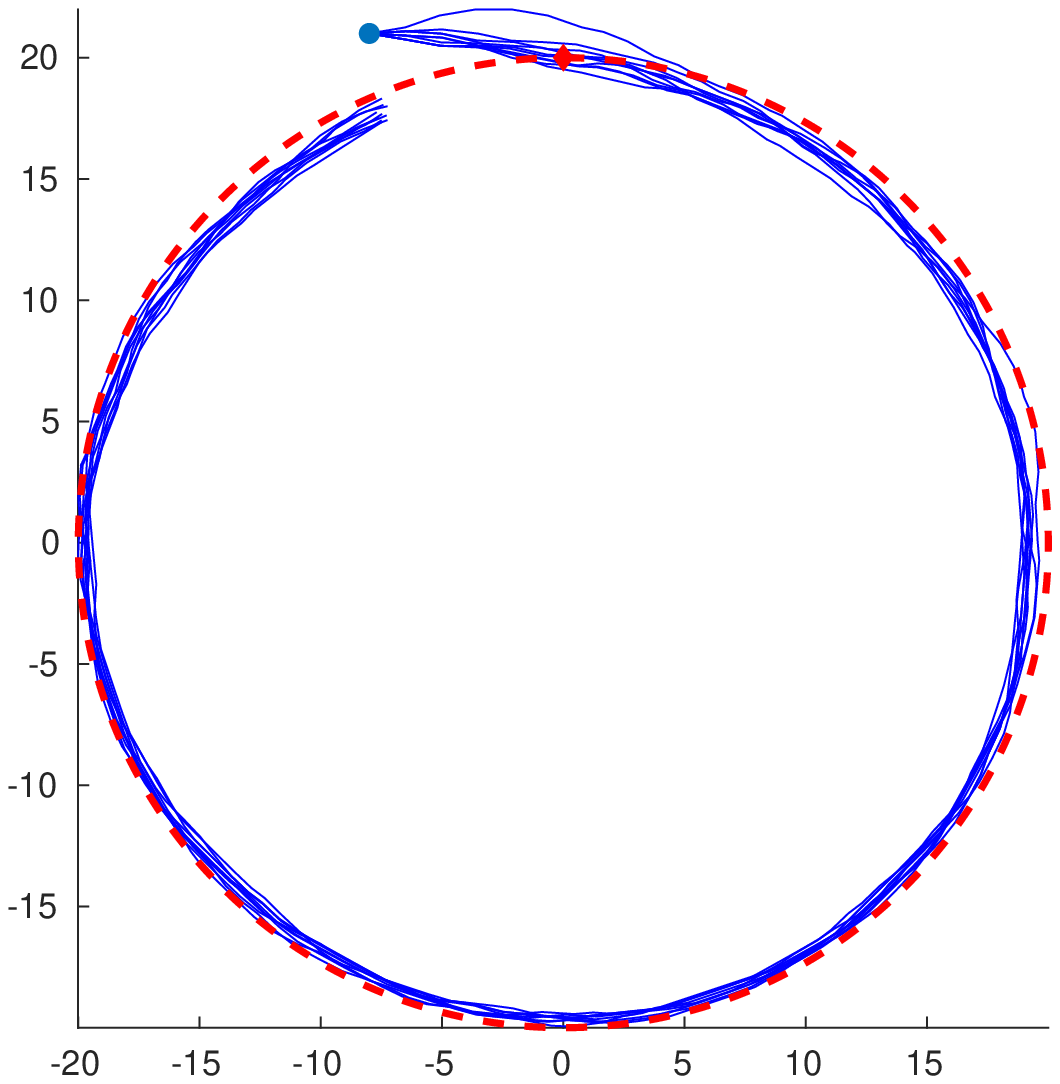}
		\caption{$\sigma_0 = 0.01$}
	\end{subfigure}
	~ \hspace{0.03\textwidth}
	\begin{subfigure}[t]{0.5\textwidth}
		\centering
		\includegraphics[scale=0.6]{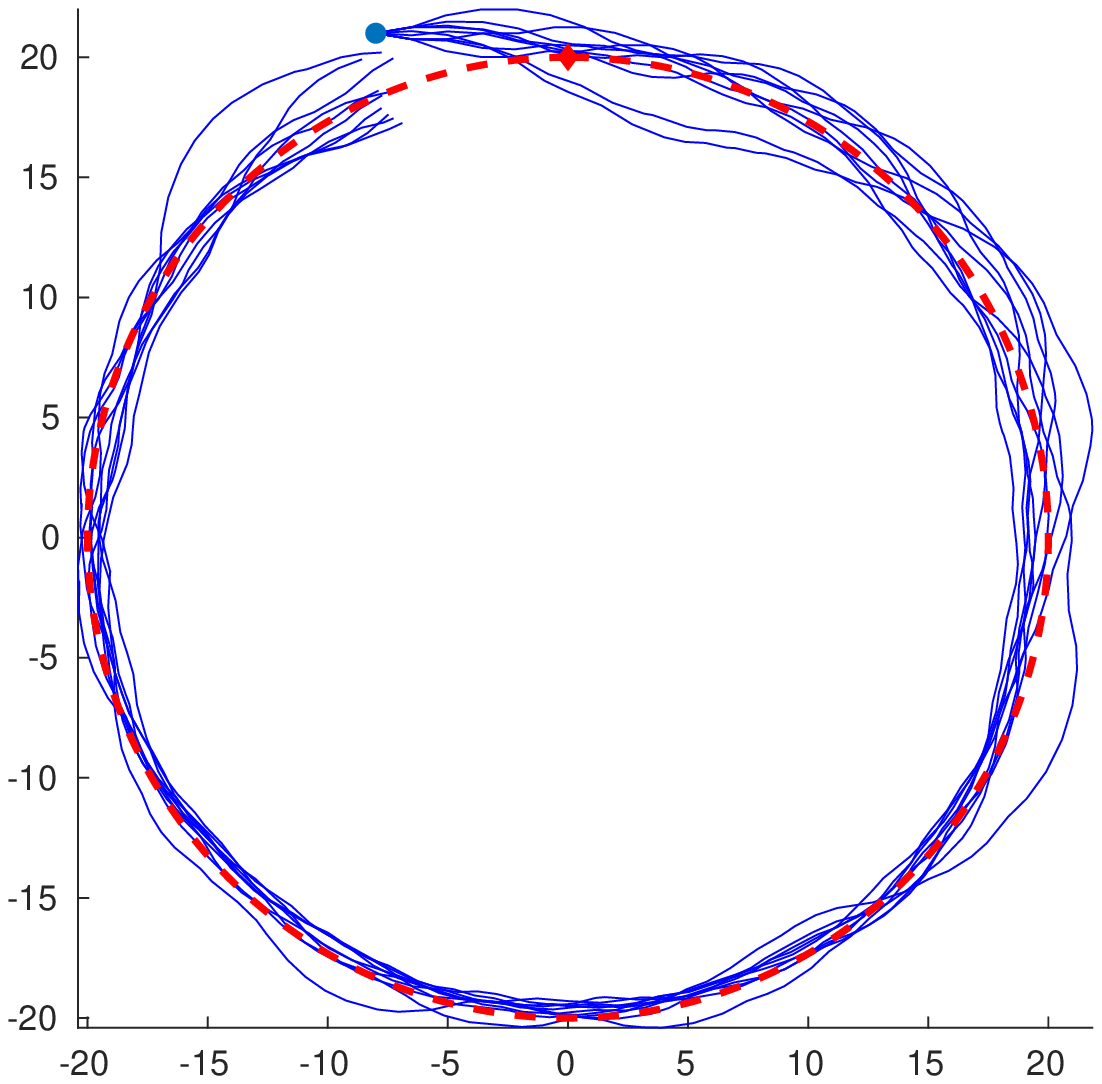}
		
		\caption{$\sigma_0 = 0.05$}
	\end{subfigure}
\\	
%
	\begin{subfigure}[t]{0.5\textwidth}
	\centering
	\includegraphics[scale=0.6]{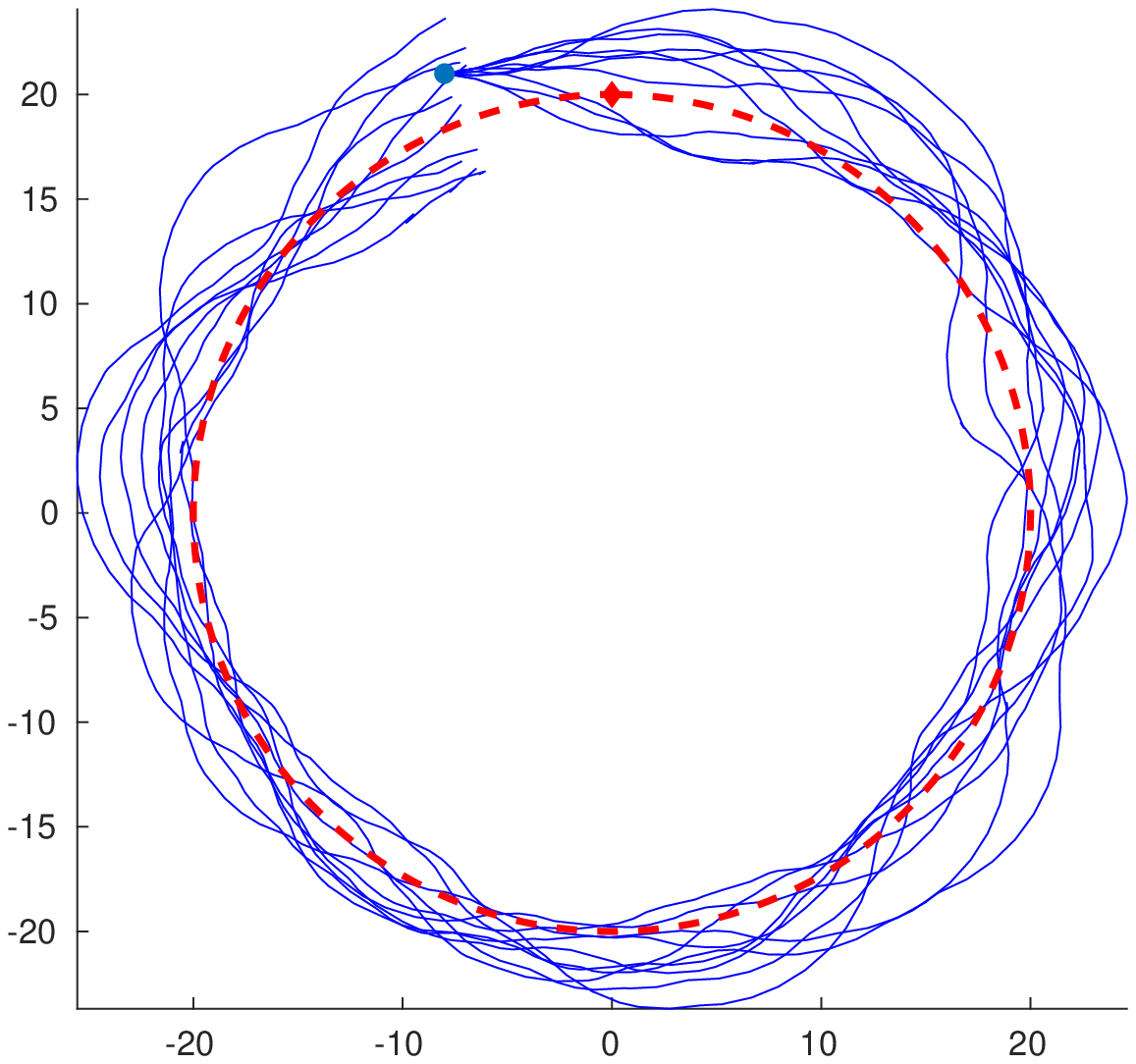}
	
	\caption{$\sigma_0 = 0.2$}
\end{subfigure}
%
	\caption{Plot of the trajectory tracking when the leading fish is swimming in circle, clockwisely. $N_{\text{exp}} = 10$ trials/realizations have been considered.}
	\label{fig:traj_track_circle}
\end{figure}

\noindent The algorithm $\texttt{MUSIC\_dipoleSearch}$ employs a multi-position dipole search that uses $M$ positions in between.  In our numerical simulations, in order to have a real-time tracking procedure, we have set $M = 5$. We have set $\theta_{\max} = 10^{-1}$.

\subsection{Electro-sensing}

While the fish can sense the presence of other conspecifics at some distance, the range for the active electrosensing is much more short \cite{kramer}.

The effectivity of the estimated position of a small dielectric target inevitably depends on the relative distances among the fish, its conspecifics and the target. However, this seems perfectly reasonable. We have to require 
\begin{enumerate}
\item The two fish do not get too close to each other;
\item The small dielectric target $D$ is in the electro-sensing range of $\mathfrak{F}_1$.
\end{enumerate}

If the above qualitative conditions are not met, there is no garantee that we can get accurate results.

\medskip

We perform many experiments to show that the MUSIC-type algorithm proposed in \cite{Am2} works under the conditions outlined above. Based on approximation \eqref{eq:dipole_approx_Im_pp} we consider the illumination vector
\[g(z) =  \left ( \nabla \wh{U}_1(z) \cdot \nabla_z \left (\frac{\p \Gamma}{\p \nu_x} \right )(x_1,z), \dots , \nabla \wh{U}_1(z) \cdot \nabla_z \left (\frac{\p \Gamma}{\p \nu_x} \right )(x_{N_r},z) \right )^T ,\]
and define the MUSIC localizer as follows:
\begin{equation}
\label{eq:MUSIC_target}
\mathcal{I}_1(z) = \frac{1}{|(I - P)\widetilde{g}(z)|} ,
\end{equation}
where $\widetilde{g} = \frac{g}{|g|}$ and $\widehat{U}_1$ is the solution to \eqref{eq:model_hat_u1_bg_test}.

\medskip
\begin{algorithm}[H]
\Input{The feedback, that is the total electric potential signal $u(x,t)$ recorded by the receptors on $\p \Omega_1$.}

\medskip

\nl Decompose the feedback $u$ into $u_1$ and $u_2$ using signal separation techniques \;

\nl $\texttt{MUSIC\_target}(u_1 |_{\p \Omega_1}, \Omega_1)$ :\\ \Indp
\nl Post-process the data $\text{Im}\, u_1 |_{\p \Omega_1}$ \;
\nl Build the SFR $\mathbf{S}_{noise}$ for the  post-processed data  \;
\nl Build and evaluate the MUSIC localizer $\mathcal{I}_1$ on the nodes of a fine uniform grid $\mathcal{G}$ in the vicinity of $\Omega_1$ \;

\nl $\wh{z}  \leftarrow \; \arg \max_{\mathcal{G}} \; \mathcal{I}_1(z)$ \;

\medskip
\medskip
\Output{An approximated position of the target.}
 \caption{SF MUSIC: Detection of a small dielectric target in the presence of another conspecific}
\label{alg:detection_algorithm2}
\end{algorithm}

%

\begin{figure}[h]
	\begin{subfigure}[t]{0.5\textwidth}
		\centering
		\includegraphics[scale=0.55]{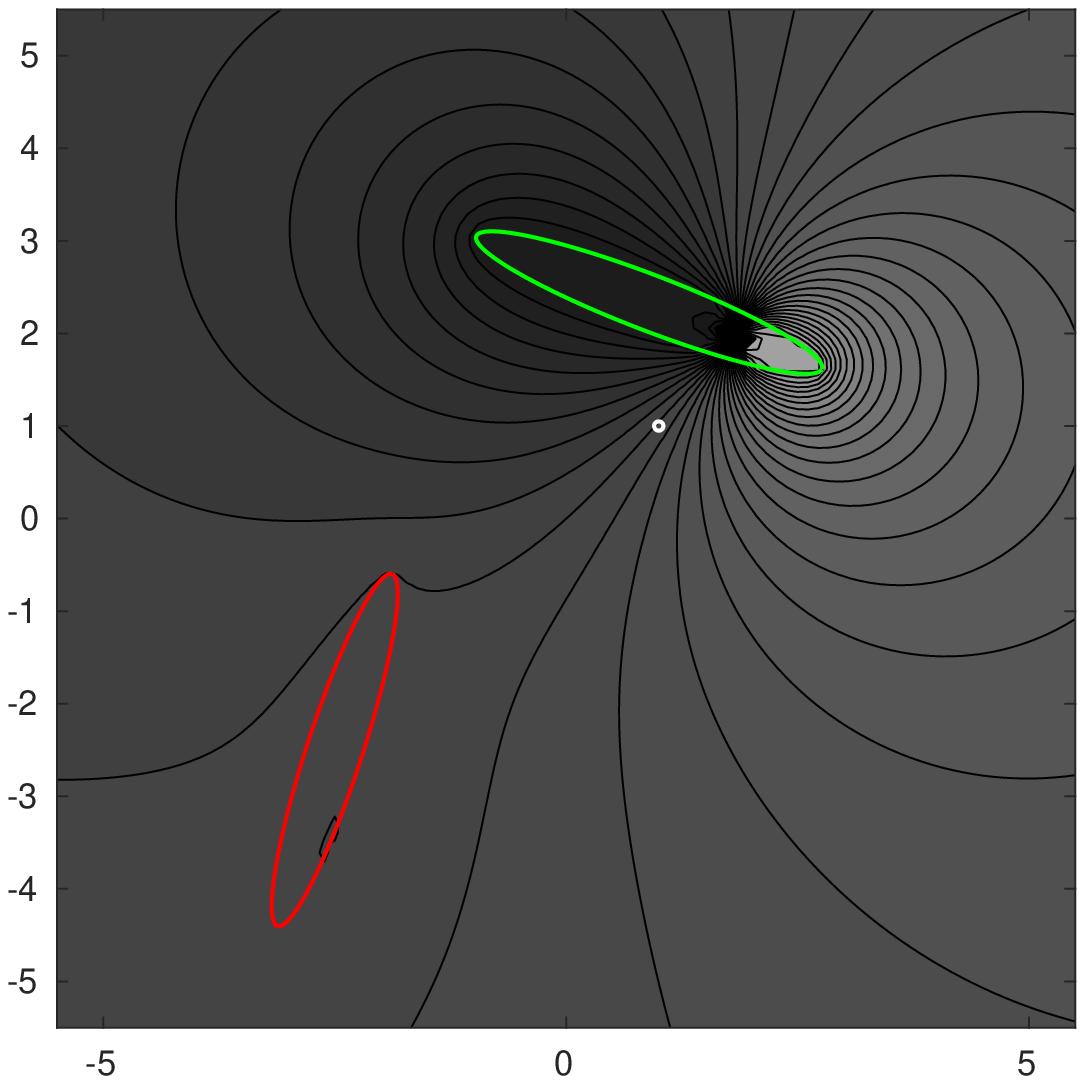}
		\caption{Plot of $\Re u_1$.}
	\end{subfigure}
	~ \hspace{0.03\textwidth}
	\begin{subfigure}[t]{0.5\textwidth}
		\centering
		\includegraphics[scale=0.55]{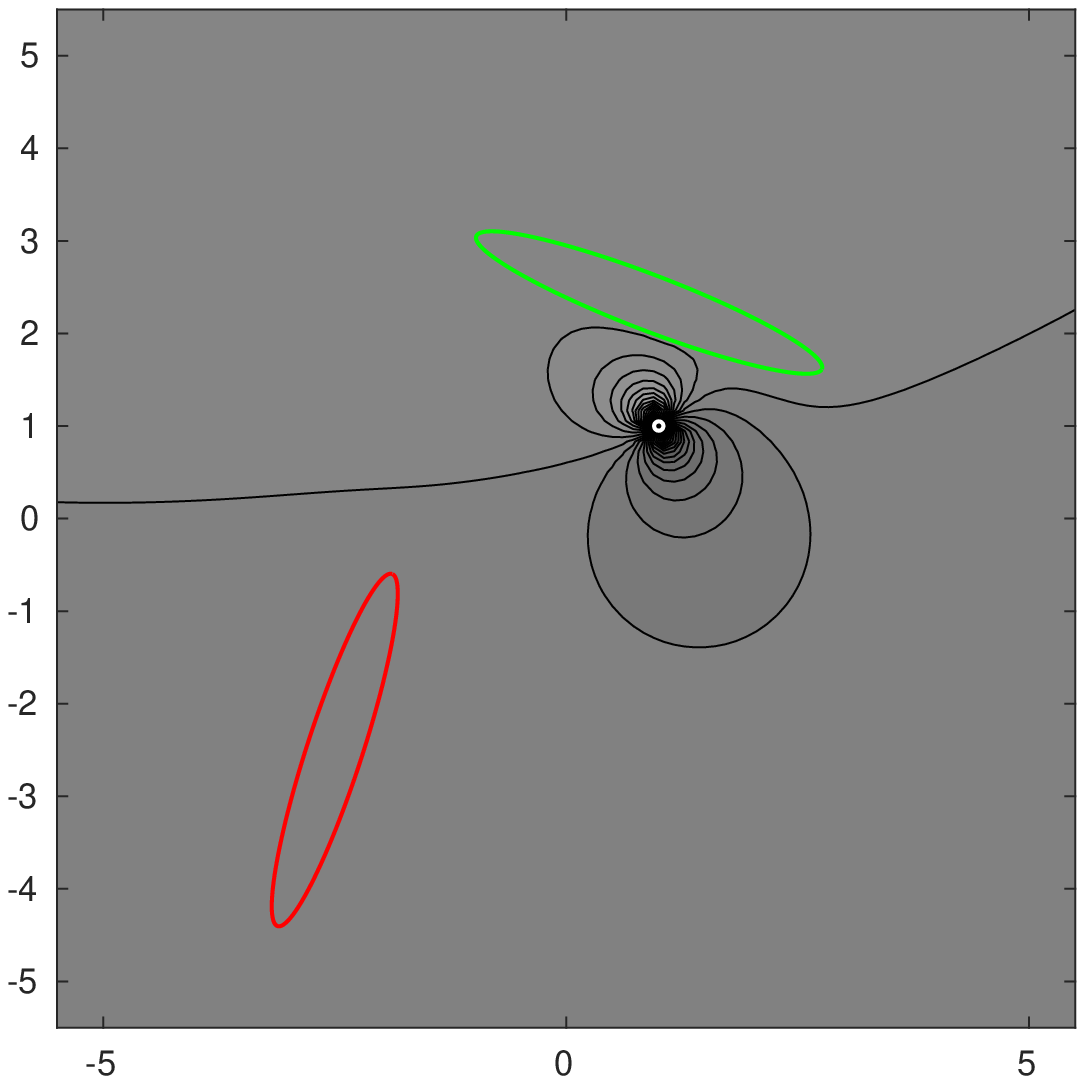}
		
		\caption{Plot of $\Im u_1$.}
	\end{subfigure}
	\caption{Plot of the isopotential lines when $\FF{2}$ (red) is passive (electrically silent) and  $\FF{1}$ (green) is active.}
	\label{fig:field}
\end{figure}

\begin{figure}[h]
    \begin{subfigure}[t]{0.5\textwidth}
        \centering
        \includegraphics[scale=0.40]{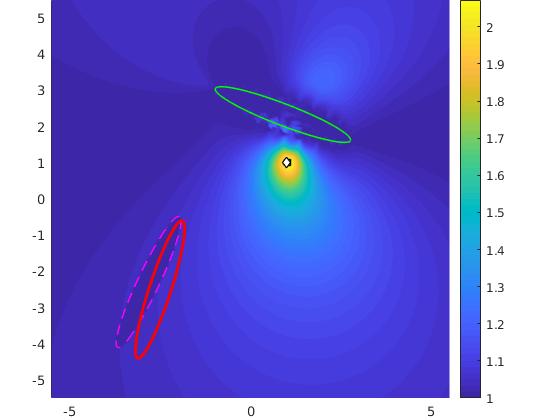}
        \caption{}
    \end{subfigure}
    ~ \hspace{0.03\textwidth}
    \begin{subfigure}[t]{0.5\textwidth}
        \centering
        \includegraphics[scale=0.40]{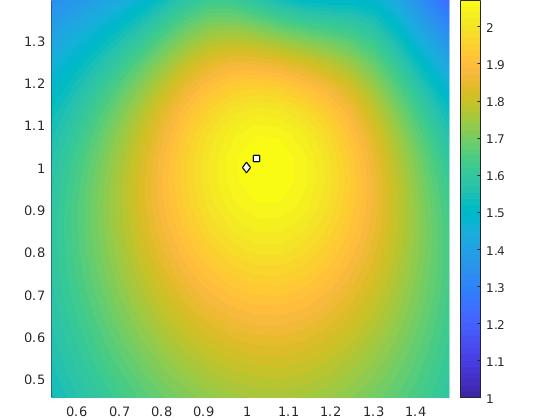}

        \caption{}
    \end{subfigure}
      \caption{Plot of the MUSIC imaging functional used in Algorithm \ref{alg:detection_algorithm2} by using $N_r = 32$ receptors and $N_f = 100$ frequencies, with noise level $\sigma_0 = 0.1$. The square and the diamond indicate the approximation of the center and the true center of the target D, respectively. $\FF{1}$ (green) can image the target despite the presence of $\FF{2}$ (red), which is estimated by applying Algorithm \ref{alg:detection_algorithm1}.}
	\label{fig:imaging_peaks}
\end{figure}

\newpage

\nocite{*}

\section{Concluding remarks}

In this paper, we have formulated the time-domain model for a shoal of weakly electric fish. We have shown how the jamming avoidance response can be interpreted within this mathematical framework and how it can be exploited to design communication systems, following strategies and active electrosensing algorithms.
In a forthcoming paper, we plan to extend our present approach to develop navigation patterns inspired by the collective behavior of the weakly electric fish.

\section{Acknowledgment}
 The author gratefully acknowledges Prof. H. Ammari for his guidance and the financial support granted by the Swiss National Foundation (grant 200021-172483).

\end{document}